\documentclass[reqno]{amsart}
\usepackage[utf8]{luainputenc}
\usepackage{xcolor}
\usepackage{pdfcolmk}
\usepackage{url}
\usepackage{enumitem}
\usepackage{amstext}
\usepackage{amsthm}
\usepackage{amssymb}
\usepackage{esint}
\PassOptionsToPackage{normalem}{ulem}
\usepackage{ulem}
\usepackage[unicode=true,pdfusetitle,
 bookmarks=true,bookmarksnumbered=false,bookmarksopen=false,
 breaklinks=false,pdfborder={0 0 0},pdfborderstyle={},backref=false,colorlinks=true]
 {hyperref}
\hypersetup{
 pdfborderstyle=,pdfborderstyle={}}

\makeatletter

\DeclareTextSymbolDefault{\textquotedbl}{T1}
\providecolor{lyxadded}{rgb}{0,0,1}
\providecolor{lyxdeleted}{rgb}{1,0,0}

\DeclareRobustCommand{\lyxsout}[1]{\ifx\\#1\else\sout{#1}\fi}

\numberwithin{equation}{section}
\numberwithin{figure}{section}
\theoremstyle{plain}
\newtheorem{thm}{\protect\theoremname}
\theoremstyle{definition}
\newtheorem{defn}[thm]{\protect\definitionname}
\theoremstyle{remark}
\newtheorem{rem}[thm]{\protect\remarkname}
\theoremstyle{plain}
\newtheorem{lem}[thm]{\protect\lemmaname}
\theoremstyle{definition}
\newtheorem{example}[thm]{\protect\examplename}
\theoremstyle{plain}
\newtheorem{cor}[thm]{\protect\corollaryname}

\usepackage{pdfsync}
\usepackage{graphics}
\usepackage{epstopdf}
\usepackage[all]{xy}
\usepackage{amscd}
\usepackage{enumitem}
\usepackage{mathtools}
\usepackage{stmaryrd}
\usepackage[utf8]{inputenc}
\setlist[enumerate]{leftmargin=*,label=(\roman*),align=left}

\makeatletter
\@namedef{subjclassname@2020}{%
  \textup{2020} Mathematics Subject Classification}
\makeatother

\newcommand{\xyR}[1]{ \makeatletter
\xydef@\xymatrixrowsep@{#1} \makeatother} 
\newcommand{\xyC}[1]{ \makeatletter
\xydef@\xymatrixcolsep@{#1} \makeatother} 
\entrymodifiers={++[ ][F]} 

\newcommand{\ra}{\longrightarrow}

\newcommand{\field}[1]{\mathbb{#1}}
\newcommand{\R}{\field{R}} 
\newcommand{\N}{\field{N}} 

\newcommand{\eps}{\varepsilon} 
\renewcommand{\phi}{\varphi}
\newcommand{\diff}[1]{\ifmmode\mathchoice{\hbox{\rm d}#1}  
 {\hbox{\rm d}#1}  
 {\scalebox{0.75}{$\hbox{\rm d}#1$}}  
 {\scalebox{0.35}{$\hbox{\rm d}#1$}}  
 \fi} 

\newcommand{\abs}[2][\empty]{\ifx#1\empty\left|#2\right|%
\else#1\vert #2 #1\vert\fi}

\newcommand{\Rtil}{\widetilde \R} 
\newcommand{\csp}[1]{{\text{\rm c}}({#1})}
\newcommand{\fcmp}{\Subset_{\text{\rm f}}}
\newcommand{\frontRise}[2]{\ifmmode\mathchoice{{\vphantom{#1}}^{\scalebox{0.6}{$#2$}}}  
 {{\vphantom{#1}}^{\scalebox{0.56}{$#2$}}}  
 {{\vphantom{#1}}^{\scalebox{0.47}{$#2$}}}  
 {{\vphantom{#1}}^{\scalebox{0.35}{$#2$}}}\fi} 
\newcommand{\RC}[1]{\frontRise{\R}{#1}\Rtil}
\newcommand{\rcrho}{\RC{\rho}}
\newcommand{\rti}{\RC{\rho}}
\newcommand{\gsf}{\frontRise{\mathcal{G}}{\rho}\mathcal{GC}^{\infty}}
\newcommand{\hyperN}[1]{	\frontRise{\N}{#1}\widetilde{\N}}

\newcommand{\hypNs}{\hyperN{\sigma}}
\newcommand{\nint}{\text{ni}}
\newcommand{\hyperlimarg}[3]{\mathchoice{\frontRise{\lim}{\raisebox{-0.05em}{$#1\hspace{-0.67em}$}}\lim_{#3\in \hyperN{#2}\,}}
{\frontRise{\lim}{#1\hspace{-0.25em}}\lim_{#3\in \hyperN{#2}\,}}
{\frontRise{\lim}{#1\hspace{-0.25em}}\lim_{#3\in \hyperN{#2}\,}}
{\frontRise{\lim}{#1\hspace{-0.25em}}\lim_{#3\in \hyperN{#2}\,}}}
\newcommand{\hyperlim}[2]{\hyperlimarg{#1}{#2}{n}}

\newcommand{\hypersumarg}[3]{\mathchoice{\frontRise{\sum}{\raisebox{-0.2em}{$#1\hspace{-0.67em}$}}\sum_{#3\in \hyperN{#2}\,}}
{\frontRise{\sum}{#1\hspace{-0.25em}}\sum_{#3\in \hyperN{#2}\,}}
{\frontRise{\sum}{#1\hspace{-0.25em}}\sum_{#3\in \hyperN{#2}\,}}
{\frontRise{\sum}{#1\hspace{-0.25em}}\sum_{#3\in \hyperN{#2}\,}}}
\newcommand{\parthypersumarg}[4]{\mathchoice{\frontRise{\sum}{\raisebox{-0.2em}{$#1\hspace{-2.2em}$}}\sum_{#3\in \hyperN{#2}_{#4}\,}}
{\frontRise{\sum}{#1\hspace{-0.25em}}\sum_{#3\in \hyperN{#2}_{#4}\,}}
{\frontRise{\sum}{#1\hspace{-0.25em}}\sum_{#3\in \hyperN{#2}_{#4}\,}}
{\frontRise{\sum}{#1\hspace{-0.25em}}\sum_{#3\in \hyperN{#2}_{#4}\,}}}
\newcommand{\hypersum}[2]{\hypersumarg{#1}{#2}{n}}

\newcommand{\subzero}{\subseteq_{0}}

\newcommand{\sbpt}[1]{#1_{\text{\rm s}}}

\newcommand{\frontRiseDown}[3]{\ifmmode\mathchoice{{\vphantom{#1}}^{\scalebox{0.6}{$#2$}}_{\scalebox{0.6}{$#3$}}}  
 {{\vphantom{#1}}^{\scalebox{0.56}{$#2$}}_{\scalebox{0.56}{$#3$}}}  
 {{\vphantom{#1}}^{\scalebox{0.47}{$#2$}}_{\scalebox{0.47}{$#3$}}}  
 {{\vphantom{#1}}^{\scalebox{0.35}{$#2$}}_{\scalebox{0.35}{$#3$}}}\fi} 
\newcommand{\RCud}[2]{\frontRiseDown{\R}{#1}{#2}\Rtil}
\newcommand{\rcrhos}{\RCud{\rho}{\sigma}_{\text{\rm s}}}

\newcommand{\gsft}{\frontRiseDown{\mathcal{G}}{\rho}{\sigma}\mathcal{GC}^{\omega}}

\newcommand{\hps}[1]{\RCud{\rho}{\sigma}\llbracket #1 \rrbracket}
\newcommand{\hpsarg}[3]{\RCud{#1}{#2}\llbracket #3 \rrbracket}
\newcommand{\rhoext}{\frontRise{\R}{\rho}\overline{\R}}
\newcommand{\radconv}[1]{\text{\rm rad}\left(#1\right)_{\rm c}}
\newcommand{\radconveps}[1]{\text{\rm rad}\left(#1\right)_{{\rm c}{\eps}}}
\newcommand{\rcrhoc}{\rcrho_{\text{\rm c}}}
\newcommand{\setconv}[2]{\frontRiseDown{\text{\rm c}}{\rho}{\sigma}\text{\rm conv}\left(\left(#1\right)_{\text{\rm c}},#2\right)}

\makeatother

\providecommand{\corollaryname}{Corollary}
\providecommand{\definitionname}{Definition}
\providecommand{\examplename}{Example}
\providecommand{\lemmaname}{Lemma}
\providecommand{\remarkname}{Remark}
\providecommand{\theoremname}{Theorem}

\begin{document}
\title{hyper-power series and generalized real analytic functions}
\author{Diksha Tiwari \and Akbarali Mukhammadiev \and Paolo Giordano}
\thanks{D.~Tiwari has been supported by grant P 30407 and P33538 of the Austrian
Science Fund FWF}
\address{\textsc{Faculty of Mathematics, University of Vienna, Austria}}
\email{diksha.tiwari@univie.ac.at}
\thanks{A.~Mukhammadiev has been supported by Grant P30407 and P33538 of
the Austrian Science Fund FWF}
\address{\textsc{Faculty of Mathematics, University of Vienna, Austria}}
\email{akbarali.mukhammadiev@univie.ac.at}
\thanks{P.~Giordano has been supported by grants P30407, P34113, P33538 of
the Austrian Science Fund FWF}
\address{\textsc{Faculty of Mathematics, University of Vienna, Austria}}
\email{paolo.giordano@univie.ac.at}
\begin{abstract}
This article is a natural continuation of the paper Tiwari, D., Giordano,
P., \emph{Hyperseries in the non-Archimedean ring of Colombeau generalized
numbers} in this journal. We study one variable hyper-power series
by analyzing the notion of radius of convergence and proving classical
results such as algebraic operations, composition and reciprocal of
hyper-power series. We then define and study one variable generalized
real analytic functions, considering their derivation, integration,
a suitable formulation of the identity theorem and the characterization
by local uniform upper bounds of derivatives. On the contrary with
respect to the classical use of series in the theory of Colombeau
real analytic functions, we can recover several classical examples
in a non-infinitesimal set of convergence. The notion of generalized
real analytic function reveals to be less rigid both with respect
to the classical one and to Colombeau theory, e.g.~including classical
non-analytic smooth functions with flat points and several distributions
such as the Dirac delta. On the other hand, each Colombeau real analytic
function is also a generalized real analytic function.
\end{abstract}

\maketitle

\section{Introduction}

In this article, the study of hyperseries in the non-Archimedean ring
of Colombeau generalized numbers (CGN), as carried out in \cite{TiGi},
is applied to the corresponding notion of hyper-power series. As we
will see, this yields results which are more closely related to classical
ones, such as, e.g.~the equality $\hypersum{\rho}{\rho}\frac{x^{n}}{n!}=e^{x}$
that holds for all $x\in\rti$ where the exponential is moderate,
i.e.~if $|x|\le\log\left(\diff{\rho}^{-R}\right)$ for some $R\in\N$.
On the other hand, we will see that classical smooth but non-analytic
functions, e.g.~smooth functions with flat points, and Schwartz distributions
like the Dirac delta, are now included in the related notion of \emph{generalized
real analytic function} (GRAF). This implies that necessarily we cannot
have a trivial generalization of the identity theorem (see e.g.~\cite[Cor.~1.2.6, 1.2.7]{KrPa02})
but, on the contrary, only a suitable sufficient condition (see Thm.~\ref{thm:identStd}
below). The notion of generalized real analytic function hence reveals
to be less rigid than the classical concept, by including a large
family of non-trivial generalized functions (e.g.~Dirac delta $\delta$,
Heaviside function $H$, but also powers $\delta^{k}$, $k\in\N$,
and compositions $\delta\circ\delta$, $\delta^{k}\circ H^{h}$, $H^{h}\circ\delta^{k}$,
etc., for $h$, $k\in\N$.

Conversely, GRAF preserve a lot of classical results: they can be
thought of as infinitely long polynomials $f(x)=\hypersum{\rho}{\sigma}a_{n}(x-c)^{n}$,
with uniquely determined coefficient $a_{n}=\frac{f^{(n)}(c)}{n!}$,
they can be added, multiplied, composed, differentiated, integrated
term by term, are closed with respect to inverse function, etc. This
lays the foundation for a potential interesting generalization of
the Cauchy-Kowalevski theorem which is able to include many non-analytic
(but generalized real analytic) generalized functions.

Concerning the theory of analytic Colombeau generalized functions,
as developed in \cite{PiScVa09} for the real case and in \cite{C1,CoGa84,Ara86,CoGa88,ObPiVa,KhSc,Ver08}
for the complex one, it is worth to mention that several properties
have been proved in both cases: closure with respect to composition,
integration over homotopic paths, Cauchy integral theorem, existence
of analytic representatives, identity theorem on a set of positive
Lebesgue measure, etc.~(cf. \cite{PiScVa09,Ver08} and references
therein). On the other hand, even if in \cite{Ver08} it is also proved
that each complex analytic Colombeau generalized functions can be
written as a Taylor series, necessarily this result holds only in
an infinitesimal neighborhood of each point. The impossibility to
extend this property to a finite neighborhood is a general drawback
of the use of ordinary series in a (Cauchy complete) non-Archimedean
framework instead of hyperseries, as explained in details in \cite{TiGi}.

We refer to \cite{MTAG} for basic notions such as the ring of Robinson-Colombeau,
subpoints, hypernatural numbers, supremum, infimum and hyperlimits,
and \cite{TiGi} for the notion of hyperseries as well as their notations
and properties. Once again, the ideas presented in the present article
can be useful to explore similar ideas in other non-Archimedean settings,
such as \cite{BL15,KhKo18,KhKo18b,BM19,Kei63,Kob96,Sham13}.

\section{hyper-power series and its basic properties}

\subsection{Definition of hyper-power series}

In the entire paper, $\rho$ and $\sigma$ are two arbitrary gauges;
only when it will be needed, we will assume a relation between them,
such as $\sigma\le\rho^{*}$ or $\sigma\ge\rho^{*}$ (see \cite{TiGi}).

A power series of \emph{real numbers} is simply ``a series of the
form $\sum_{n\in\N}a_{n}(x-c)^{n}$, where $x$, $c$, $a_{n}\in\R$
for all $n\in\N$''. Actually, this (informal) definition allows
us to consider only finite sums $\sum_{n=0}^{N}a_{n}(x-c)^{n}$, $N\in\N$,
and hence to evaluate whether convergence holds or not. A similar
approach can be used for hyper-power series (HPS) if we think at the
$\rti$-module $\rcrhos$ of sequences for hyperseries exactly as
the space where we can consider hyperfinite sums regardless of convergence.
This is the idea to define the space $\hps{x-c}$ of formal HPS:
\begin{defn}
\label{def:formalHps}Let $x$, $c\in\rti$. We say $(b_{n})_{n}\in\hps{x-c}$
if and only if there exist $\left(a_{n\eps}\right)_{n,\eps}\in\R^{\N\times I}$
and representatives $[x_{\eps}]=x$, $[c_{\eps}]=c$ such that
\begin{equation}
(b_{n})_{n}=\left[a_{n\eps}\cdot\left(x_{\eps}-c_{\eps}\right)^{n}\right]_{\text{{\rm s}}}\in\rcrhos.\label{eq:defHPS}
\end{equation}
\end{defn}

\noindent Elements of $\hps{x-c}$ are called \emph{formal} HPS because
here we are not considering their convergence. In other words, a formal
HPS is a hyper series (i.e.~an equivalence class $(b_{n})_{n}\in\rcrhos$
in the space of sequences for hyperseries) of the form $\left[a_{n\eps}\cdot\left(x_{\eps}-c_{\eps}\right)^{n}\right]_{\text{{\rm s}}}$.
\begin{rem}
~
\begin{enumerate}
\item We explicitly note that $x-c$ is not an indeterminate, like in the
case of formal power series $\R\llbracket x\rrbracket$, but a generalized
number of $\rcrho$. For example, in Lem\@.~\ref{lem:indepRepr}
below, we will prove that if $x-c=y-d$, then $\hps{x-c}=\hps{y-d}$.
\item On the contrary with respect to the case of real numbers, being a
formal HPS, i.e.~an element of $\hps{x-c}$, depends on the interplay
of the two gauges $\rho$ and $\sigma$: take e.g.~$a_{n}=\frac{1}{n^{2}}$
and $x-c=2$, so that for all $N\in\hypNs$ we have $\sum_{n=1}^{N}a_{n}(x-c)^{n}\ge\sum_{n=0}^{N}\frac{1}{n}\sim\log(N)$.
Therefore, taking e.g.~$\sigma_{\eps}=\exp\left(-\exp\left(\frac{1}{\rho_{\eps}}\right)\right)$
and $N_{\eps}:=\text{int}\left(\sigma_{\eps}\right)$, we have that
$\left(\log N_{\eps}\right)\notin\R_{\rho}$ and hence we cannot even
consider hyperfinite sums of this form. Informally stated, for this
gauge $\sigma$, we have that $\hypersum{\rho}{\sigma}\frac{1}{n^{2}}2^{n}$
is \emph{not} a formal HPS, i.e.~even before considering its convergence
or not, we cannot compute $\sigma$-hyperfinite sums and get a number
in $\rcrho$.
\item In \cite{TiGi}, we proved that if $x$ is finite, then $\left[\frac{x^{n}}{n!}\right]_{\text{{\rm s}}}\in\hps{x}$
is a formal HPS for all gauges $\rho$, $\sigma$. In Sec.~\ref{exa:Exponential},
we will prove that $\left[\frac{\left(\diff{\rho}^{-1}\right)^{n}}{n!}\right]_{\text{{\rm s}}}\notin\hpsarg{\rho}{\rho}{\diff{\rho}^{-1}}$;
on the other hand, we will also see that if $x\le\log\left(\diff\rho^{-N}\right)$
and $\diff\sigma^{Q}\le\diff\rho^{N}$ for some $Q\in\N$, then $\left[\frac{x^{n}}{n!}\right]_{\text{{\rm s}}}\in\hps{x}$
is a formal HPS.
\end{enumerate}
\end{rem}

The previous Def.~\ref{def:formalHps} sets immediate problems concerning
independence of representatives: every time we start from $\left[a_{n\eps}\right]\in\rcrho$,
for all $n\in\N$, $[x_{\eps}]=x$, $[c_{\eps}]=c$ and we have that
$\left[a_{n\eps}\cdot\left(x_{\eps}-c_{\eps}\right)^{n}\right]_{\text{{\rm s}}}\in\rcrhos$,
we can consider whether the corresponding formal HPS $\hypersum{\rho}{\sigma}[a_{n\eps}]\cdot\left([x_{\eps}]-[c_{\eps}]\right)^{n}$
converges or not. On the other hand, we also have to prove that it
is well-defined, i.e.~that taking different representatives $\left[\bar{a}_{n\eps}\right]=\left[a_{n\eps}\right]$,
$[\bar{x}_{\eps}]=x$, $[\bar{c}_{\eps}]=c$, we have $\left[a_{n\eps}\cdot\left(x_{\eps}-c_{\eps}\right)^{n}\right]_{\text{{\rm s}}}=\left[\bar{a}_{n\eps}\cdot\left(\bar{x}_{\eps}-\bar{c}_{\eps}\right)^{n}\right]_{\text{{\rm s}}}$.
However, from \cite[Sec.~2]{TiGi} it follows that we can have $x-c=1$
and $\left[a_{n\eps}\right]=\left[\bar{a}_{n\eps}\right]=0$ for all
$n\in\N$, but
\[
\left[\sum_{n=0}^{N_{\eps}}a_{n\eps}(x_{\eps}-c_{\eps})^{n}\right]\ne\left[\sum_{n=0}^{N_{\eps}}\bar{a}_{n\eps}(x_{\eps}-c_{\eps})^{n}\right].
\]
This means that $(b_{n})_{n}:=\left[a_{n\eps}(x_{\eps}-c_{\eps})^{n}\right]_{\text{s}}$
and $(\bar{b}_{n})_{n}:=\left[\bar{a}_{n\eps}(x_{\eps}-c_{\eps})^{n}\right]_{\text{s}}$
yield two different formal HPS (see \cite[Thm.~4]{TiGi}) and hence,
in general, the operation
\[
\left((a_{n\eps})_{n,\eps},(x_{\eps}),(c_{\eps})\right)\in\rti^{\N}\times\rti^{2}\mapsto(b_{n})_{n}:=\left[a_{n\eps}(x_{\eps}-c_{\eps})^{n}\right]_{\text{s}}\in\rcrhos
\]
is not well-defined. The problem can also be addressed differently:
what notion of equality do we have to set on a suitable subring of
$\R^{\N\times I}$ so as to have independence on representatives?
This notion of equality naturally emerges in proving that the following
definition of radius of convergence is well-defined (see Lem.~\ref{lem:radConvWellDef}).
What subring we need to consider arises from the idea to include $\left(\frac{\delta_{\eps}^{(n)}(0)}{n!}\right)_{n,\eps}$in
it, where $\delta=\left[\delta_{\eps}(-)\right]$ is a suitable embedding
of Dirac's delta function (see Example \ref{rem:radConv}.\ref{enu:deltaRadConv}).

\subsection{Radius of convergence}

The idea to define the radius of convergence corresponding to coefficients
$(a_{n\eps})_{n,\eps}\in\R^{\N\times I}$ is that it does not matter
if
\[
\left(\limsup_{n\to+\infty}\left|a_{n\eps}\right|^{1/n}\right)^{-1}\in\R\cup\{+\infty\}
\]
yields a non $\rho$-moderate net (for example for $\eps\in L\subzero I$)
because this case would intuitively identify a radius of convergence
larger than any infinite number in $\rcrho$:
\begin{defn}
\label{def:radCon}~
\begin{enumerate}
\item \label{enu:Rext}Let $\overline{\R}:=\R\cup\{-\infty,\infty\}$ be
the extended real number system with the usual (partially defined)
operations. We set $\rhoext:=\overline{\R}^{I}/\sim_{\rho}$, where
for arbitrary $(x_{\eps})$, $(y_{\eps})\in\overline{\R}^{I}$, as
usual we define
\[
(x_{\eps})\sim_{\rho}(y_{\eps})\quad:\iff\quad\forall q\in\N\,\forall^{0}\eps:\ \left|x_{\eps}-y_{\eps}\right|\le\rho_{\eps}^{q}.
\]
In $\rhoext$, we can also consider the standard order relation
\[
x\le y\quad:\iff\quad\exists[x_{\eps}]=x,[y_{\eps}]=y\,\forall^{0}\eps:\ x_{\eps}\le y_{\eps}.
\]
Note that $\left(\rhoext\setminus\{-\infty\},+,\le\right)$ is an
ordered group but, since we are considering arbitrary nets $\overline{\R}^{I}$,
the set $\rhoext$ is not a ring: e.g.~$+\infty\cdot0$ is still
undefined and $+\infty\cdot[z_{\eps}]=[+\infty]$ for all $(z_{\eps})\in\R_{>0}^{I}$.
\item \label{enu:coeffHPS}Moreover, we denote by $\rcrhoc:=\left(\R^{\N\times I}\right)_{\rho}/\simeq_{\rho}$
the quotient ring of \emph{coefficients for HPS}, where 
\begin{equation}
\left(a_{n\eps}\right)_{n,\eps}\in\left(\R^{\N\times I}\right)_{\rho}:\!\iff\exists Q,R\in\N\,\forall^{0}\eps\,\forall n\in\N:\ \left|a_{n\eps}\right|\le\rho_{\eps}^{-nQ-R}\label{eq:weakMod}
\end{equation}
is the ring of \emph{weakly $\rho$-moderate nets}, and
\begin{equation}
\left(a_{n\eps}\right)_{n,\eps}\simeq_{\rho}\left(\bar{a}_{n\eps}\right)_{n,\eps}:\!\iff\forall q,r\in\N\,\forall^{0}\eps\,\forall n\in\N:\ \left|a_{n\eps}-\bar{a}_{n\eps}\right|\le\rho_{\eps}^{nq+r}\label{eq:strongEq}
\end{equation}
(in this case, we say that these two nets are \emph{strongly $\rho$-equivalent}).
Equivalence classes of $\rcrhoc$ are denoted by $(a_{n})_{{\rm c}}:=\left[a_{n\eps}\right]_{\text{c}}\in\rcrhoc$.
\item \label{enu:radConv}Finally, if $(a_{n})_{{\rm c}}=\left[a_{n\eps}\right]_{\text{c}}\in\rcrhoc$,
then we set $\radconveps{a_{n}}:=r_{\eps}.$ and $\radconv{a_{n}}=:[r_{\eps}]\in\rhoext$,
where
\begin{equation}
r_{\eps}:=\left(\limsup_{n\to+\infty}\left|a_{n\eps}\right|^{1/n}\right)^{-1}\in\R\cup\{+\infty\}.\label{eq:r_eps}
\end{equation}
\end{enumerate}
\end{defn}

\noindent In the following lemma, we prove that $\radconv{a_{n}}$
is well-defined:
\begin{lem}
\label{lem:radConvWellDef}Let $(a_{n})_{{\rm c}}=\left[a_{n\eps}\right]_{\text{{\rm c}}}=\left[\bar{a}_{n\eps}\right]_{\text{{\rm c}}}\in\rcrhoc$.
Define $r_{\eps}$ as in \eqref{eq:r_eps} and similarly define $\bar{r}_{\eps}$
using $\bar{a}_{n\eps}$. Then $(r_{\eps})\sim_{\rho}(\bar{r}_{\eps})$,
and hence $[r_{\eps}]=[\bar{r}_{\eps}]$ in $\rhoext$.
\end{lem}

\begin{proof}
For all $\eps\in I$ and all $n\in\N_{>0}$, we have $\left|\bar{a}_{n\eps}\right|^{1/n}\le\left(\left|\bar{a}_{n\eps}-a_{n\eps}\right|+\left|a_{n\eps}\right|\right)^{1/n}$.
The binomial formula yields $(x+y)\le\left(x^{1/n}+y^{1/n}\right)^{n}$
for all $x$, $y\in\R_{\ge0}$, so that $\left|\bar{a}_{n\eps}\right|^{1/n}\le\left|\bar{a}_{n\eps}-a_{n\eps}\right|^{1/n}+\left|a_{n\eps}\right|^{1/n}$.
Setting $r=0$ in \eqref{eq:strongEq}, for all $q\in\N$ and for
$\eps$ small we have
\[
\forall n\in\N:\ \left|a_{n\eps}-\bar{a}_{n\eps}\right|\le\rho_{\eps}^{nq}.
\]
Therefore, for the same $\eps$ we get $\left|\bar{a}_{n\eps}\right|^{1/n}\le\rho_{\eps}^{q}+\left|a_{n\eps}\right|^{1/n}$.
Taking the limit superior we obtain $\limsup_{n\to+\infty}\left|\bar{a}_{n\eps}\right|^{1/n}\le\rho_{\eps}^{q}+\limsup_{n\to+\infty}\left|a_{n\eps}\right|^{1/n}$.
Inverting the role of $(a_{n\eps})_{n,\eps}$ and $(\bar{a}_{n\eps})_{n,\eps}$
we finally obtain
\[
\forall^{0}\eps:\ -\rho_{\eps}^{q}\le\limsup_{n\to+\infty}\left|a_{n\eps}\right|^{1/n}-\limsup_{n\to+\infty}\left|\bar{a}_{n\eps}\right|^{1/n}\le\rho_{\eps}^{q},
\]
which proves the claim.
\end{proof}
\begin{rem}
~\label{rem:radConv}
\begin{enumerate}
\item \label{enu:whyR}If $(a_{n})_{{\rm c}}=\left[a_{n\eps}\right]_{\text{{\rm c}}}\in\rcrhoc$,
then for each fixed $n\in\N$, we have that $[\left(a_{n\eps}\right)_{\eps}]\in\rti$,
i.e.~the net $\left(a_{n\eps}\right)_{\eps}$ is $\rho$-moderate.
This is the main motivation to consider the exponent ``$-R$'' in
\eqref{eq:weakMod} (recall that in our notation $0\in\N$): without
the term ``$-R$'', the only possibility to have $(a_{n})_{c}\in\rhoext$
is that $|a_{0}|\le1$, which is an unnecessary limitation. Similarly,
we can motivate why we are considering the quantifier ``$\forall n\in\N$''
in the same formula (instead of, e.g., ``$\exists N\in\N\,\forall n\in\N_{\ge N}$'').
The proof of the next Lem.~\ref{lem:indepRepr} will motivate why
in \eqref{eq:weakMod} we consider the uniform property ``$\forall^{0}\eps\,\forall n\in\N$''
and not ``$\forall n\in\N\,\forall^{0}\eps$''.
\item \label{enu:rcrhoSubsetRhoext}Note that $\rcrho\subseteq\rhoext$
because the notion of equality $\sim_{\rho}$ in the two quotient
sets is the same and because if $(x_{\eps})$ is $\rho$-moderate
and $(x_{\eps})\sim_{\rho}(y_{\eps})$, then also $(y_{\eps})$ is
$\rho$-moderate.
\item Condition \eqref{eq:weakMod} of being weakly $\rho$-moderate represents
a constrain on what coefficients $a_{n}$ we can consider in a hyperseries.
For example, if $(a_{n})_{n\in\N}$ is a sequence of real numbers
satisfying $|a_{n}|\le p(n)$, where $p\in\R[x]$ is a polynomial,
then $p(n)\le\rho_{\eps}^{-nQ}$ for all $\eps$ sufficiently small
and for all $n\in\N$ if $Q\ge\max\left(1,\max\left\{ -\frac{\log n}{p(n)\log\rho_{\eps}}\mid n<N_{1}\right\} \right)$,
where $\frac{\log n}{P(n)}\le1$ for all $n\ge N_{1}$ and $-\frac{1}{\log\rho_{\eps}}\le1$.
Hence $\left(a_{n}\right)_{n,\eps}\in\left(\R^{\N\times I}\right)_{\rho}$
is weakly $\rho$-moderate. On the contrary, we cannot have $n^{n}\le\rho_{\eps}^{-nQ-R}=\rho_{\eps}^{-R}\left(\frac{1}{\rho_{\eps}^{Q}}\right)^{n}$
for all $n\in\N$. Similarly $(n!)_{n\in\N}$ is not weakly $\rho$-moderate
and hence our theory does not apply to a ``hyperseries'' of the
form $\hypersum{\rho}{\sigma}n!\cdot x^{n}$. On the other hand, in
Lem.~\ref{thm:radConvTool}.\ref{enu:radConvZero} we will show that,
as a consequence of considering only weakly moderate coefficients,
the radius of convergence of our hyperseries is always strictly positive.
\item \label{enu:nonModRadConv}Let $a_{n\eps}=\rho_{\eps}^{\frac{n+1}{\eps}}$,
so that $\left[a_{n\eps}\right]_{\text{c}}=0$. The corresponding
radius of convergence is $r_{\eps}=\lim_{n\to+\infty}|a_{n\eps}|^{1/n}=\rho_{\eps}^{1/\eps}$
which is not $\rho$-moderate. In general, if $\radconv{a_{n}}=[r_{\eps}]=:r\in\rhoext$,
we can have different behavior on different subpoints, e.g.~$r|_{L_{1}}=+\infty$,
$r|_{L_{2}}\in\rti$, $r|_{L_{3}}$ non $\rho$-moderate, etc., where
$L_{i}\subzero I$. This behavior is studied in Lem.~\ref{thm:radConvTool}
below.
\item \label{enu:deltaRadConv}Let $\mu:=\mathcal{F}^{-1}(\beta)\in\mathcal{S}(\R)$
be a Colombeau mollifier defined as the inverse Fourier transform
of a smooth, supported in $[-1,1]_{\R}$, even bump function $0\le\beta\le1$
which identically equals $1$ in a neighborhood of $0$ (see e.g.~\cite{GKOS}).
Let $i_{\R}^{b}$ be the embedding of Schwartz distributions into
generalized smooth functions (GSF) defined by $\mu$ and by the infinite
number $b\in\rti$ (see e.g.~\cite{TI}). The Schwartz's Paley-Wiener
theorem implies that $\mu$ is an entire function and we know that
if $\diff\rho^{-Q}\ge b=[b_{\eps}]\ge\diff\rho^{-R}$, for some $Q$,
$R\in\R_{>0}$, then the embedding of Dirac delta $\delta:=\iota_{\R}^{b}\left(\phi\mapsto\phi(0)\right)\in\gsf(\rti,\rti)$
is defined by the net $\delta_{\eps}(x)=b_{\eps}\mu(b_{\eps}x)$ (see
e.g.~\cite{TI}). For $n\in\N$, we have $\mu^{(n)}(0)=\frac{1}{2\pi}\int\beta(x)(ix)^{n}\,\diff x=0$
if $n$ is odd and $\left|\mu^{(n)}(0)\right|\le\frac{1}{2\pi}\left[\frac{x^{n+1}}{n+1}\right]_{-1}^{1}\le1$
if $n$ if even. Thereby $\left|\frac{\delta_{\eps}^{(n)}(0)}{n!}\right|=\left|\frac{\mu^{(n)}(0)}{n!}b_{\eps}^{n+1}\right|\le\frac{1}{n!}\rho_{\eps}^{-nQ-Q}\le\rho_{\eps}^{-nQ-Q}$.
This inequality shows that $\left(\frac{\delta^{(n)}(0)}{n!}\right)_{\text{c}}\in\rcrhoc$
and motivates our definition of weakly $\rho$-moderate nets. The
corresponding radius of convergence is $r_{\eps}^{-1}=\limsup_{n\to+\infty}\left|\frac{\delta_{\eps}^{(n)}(0)}{n!}\right|^{1/n}=\limsup_{n\to+\infty}b_{\eps}^{1+1/n}\left|\frac{\mu^{(n)}(0)}{n!}\right|^{1/n}=b_{\eps}\cdot0=0$,
i.e.~$\radconv{\frac{\delta^{(n)}(0)}{n!}}=+\infty$.
\item Let $(a_{n})_{{\rm c}}=[a_{n\eps}]_{{\rm c}}\in\rcrhoc$, and assume
that for all $\eps$ there exists $r_{\eps}:=\left(\lim_{n\to+\infty}|a_{n\eps}|^{1/n}\right)^{-1}$
such that $r:=[r_{\eps}]\in\rti$. Then from \cite[Thm.~28]{MTAG},
for some gauge $\sigma\leq\rho$ we have $\hyperlim{\rho}{\sigma}|a_{n}|^{1/n}=\frac{1}{r}$
and $r=\radconv{a_{n}}\in\rcrho$. In Cor.~\ref{lem:sup_eq_rad},
we will see the relationship between our definition of radius of convergence
and the least upper bound of all the radii where the HPS converges.
\end{enumerate}
\end{rem}

In the following lemma, we show that $\rcrhoc$ is a ring:
\begin{lem}
\label{lem:rcrhocRing}With pointwise operations, $\rcrhoc$ is a
quotient ring.
\end{lem}

\begin{proof}
Actually, the result follows from \cite[Thm.~3.6]{GiLu16} because
the set
\[
\mathcal{B}:=\{\left(\rho_{\eps}^{-nQ-R}\right)_{n,\eps}\in\R^{\N\times I}\mid Q,\ R\in\N\}
\]
is an asymptotic gauge with respect to the order $(n,\eps)\le(\bar{n},\bar{\eps})$
if and only if $\eps\le\bar{\eps}$. However, an independent proof
follows the well-known lines of the corresponding proof for the ring
$\rti$, and depends on the following properties of $\mathcal{B}$:
\begin{enumerate}[label=(\alph*)]
\item $\forall p,q\in\mathcal{B}\,\exists r,s\in\mathcal{B}:\ p+q\le r,\ p\cdot q\le s$;
\item $\forall p\in\mathcal{B}\,\exists r,s\in\mathcal{B}:\ r^{-1}+s^{-1}\le p^{-1}$;
\item $\forall p,q,r\in\mathcal{B}\,\exists u,v\in\mathcal{B}:\ u^{-1}\cdot q+v^{-1}\cdot r\le p^{-1}$,
\end{enumerate}
where $p=(p_{n\eps})_{n,\eps}\le(q_{n\eps})_{n,\eps}=q$ means $\forall^{0}\eps\,\forall n\in\N:\ p_{n\eps}\le q_{n\eps}$.
\end{proof}
The following lemma represents a useful tool to deal with the radius
of convergence. It essentially states that the radius of convergence
equals $+\infty$ on some subpoint, or it is moderate on some subpoint
or it is greater than any power $\diff\rho^{-P}$.
\begin{thm}
\label{thm:radConvTool}Let $\left(a_{n}\right)_{\text{\emph{c}}}\in\rcrhoc$
and $r=[r_{\eps}]=\radconv{a_{n}}\in\rhoext$, then we have
\begin{enumerate}
\item \label{enu:radConvZero}$r>0$.
\item \label{enu:radConvInfty}$r<+\infty$ or $r\sbpt{=}+\infty$.
\item \label{enu:radConvAlternatives}If $r<+\infty$, then the following
alternatives hold
\begin{enumerate}[label=(\alph*)]
\item \label{enu:radConvAltNotMod}$\forall P\in\N:\ r>\diff\rho^{-P}$
or
\item \label{enu:radConvAltMod}setting
\begin{align}
 & \left[r\le\rho^{-P}\right]:=\left\{ \eps\mid r_{\eps}\le\rho_{\eps}^{-P}\right\} =:L_{P}\nonumber \\
 & P_{\text{m}}:=\min\left\{ P\in\N\mid\left[r\le\rho^{-P}\right]\subzero I\right\} \label{eq:P_m}
\end{align}
we have
\begin{enumerate}[label=(b.\arabic*)]
\item \label{enu:I_as_unionL_P}$I=\bigcup_{P\in\N}\left[r\le\rho^{-P}\right]$;
\item \label{enu:PgeP_m}$\forall P\ge P_{\text{m}}:\ \left[r\le\rho^{-P}\right]\subzero I,\ r\le_{L_{P}}\diff\rho^{-P}$;
\item \label{enu:P<P_m}$\forall P<P_{\text{m}}:\ \diff\rho^{-P}\le r$;
\item \label{enu:P_m=00003D0}If $P_{\text{m}}=0$ and $L_{0}^{c}\subzero I$,
then $1\le_{L_{0}^{c}}r$; if $L_{0}^{c}\not\subzero I$, then $r\le1$.
\end{enumerate}
\end{enumerate}
\item \label{enu:radConvDichotomy}Assume that for all $L\subzero I$, the
following implication holds
\begin{equation}
\left(\exists Q\in\N:\ r\le_{L}\diff\rho^{-Q}\right)\text{ or }\left(\forall Q\in\N:\ r>_{L}\diff\rho^{-Q}\right)\ \Rightarrow\ \forall^{0}\eps\in L:\ \mathcal{P}\left\{ r_{\eps}\right\} .\label{eq:implication}
\end{equation}
Then $\forall^{0}\eps:\ \mathcal{P}\left\{ r_{\eps}\right\} $, i.e.~the
property $\mathcal{P}\left\{ r_{\eps}\right\} $ holds for all sufficiently
small $\eps$.
\item \label{enu:radConvModLe}If $q\in\rcrho$ and $q<r$, then $\exists s\in\rti:\ q<s\le r$.
\end{enumerate}
\end{thm}

\begin{proof}
\ref{enu:radConvZero}: Assume that $\left|a_{n\eps}\right|\le\rho_{\eps}^{-nQ-R}$
for all $\eps\le\eps_{0}$ and for all $n\in\N$. Then $\limsup_{n\to+\infty}\left|a_{n\eps}\right|^{1/n}\le\lim_{n\to+\infty}\rho_{\eps}^{-Q-\frac{R}{n}}=\rho_{\eps}^{-Q}$,
i.e.~$r_{\eps}\ge\rho_{\eps}^{Q}$.

\ref{enu:radConvInfty}: Set $L:=\left\{ \eps\mid r_{\eps}=+\infty\right\} $.
If $L\subzero I$, then $r=_{L}+\infty$. Otherwise $(0,\eps_{0}]\cap L=\emptyset$
for some $\eps_{0}$, i.e.~$r_{\eps}<+\infty$ for all $\eps\le\eps_{0}$.

\ref{enu:radConvAlternatives}: Since we assume that $r<+\infty$,
without loss of generality we can take $r_{\eps}<+\infty$ for all
$\eps$. We also assume that \ref{enu:radConvAltNotMod} is false,
i.e.~$r\le_{M}\diff\rho^{-\bar{P}}$ for some $\bar{P}\in\N$ and
some $M\subzero I$. We first prove \ref{enu:I_as_unionL_P}: take
$\eps\in\bigcap_{P\in\N}\left[r>\rho^{-P}\right]$, then $r_{\eps}>\rho_{\eps}^{-P}$
for all $P\in\N$, so that $r_{\eps}=+\infty$ for $P\to+\infty$,
and this is not possible. We also note that $\left[r\le\rho^{-P}\right]\subseteq\left[r\le\rho^{-Q}\right]$
for all $Q\ge P$. From $M\subzero I$ and $r\le_{M}\diff\rho^{-\bar{P}}$,
we have $(0,\eps_{0}]\cap M\subseteq\left[r\le\rho^{-(\bar{P}+1)}\right]\subzero I$,
and hence definition \eqref{eq:P_m} yields $P_{\text{m}}\in\N$ and
also proves \ref{enu:PgeP_m}. For all $P\in\N_{<P_{\text{m}}}$,
we hence have $\left[r\le\rho^{-P}\right]\not\subzero I$, i.e.~$(0,\eps_{P}]\subseteq\left[r>\rho^{-P}\right]$
for some $\eps_{P}$. This implies $\diff\rho^{-P}\le r$ and proves
\ref{enu:P<P_m}. Finally, if $P_{\text{m}}=0$ and $L_{P_{\text{m}}}^{c}=L_{0}^{c}\subzero I$,
then $1\le_{L_{0}^{c}}r$ because $L_{0}^{c}=\left[r>1\right]$. If
$L_{0}^{c}\not\subzero I$, then $(0,\eps_{0}]\subseteq L_{0}$ for
some $\eps_{0}$, i.e.~$r\le1$.

\ref{enu:radConvDichotomy}: By contradiction, assume that $\neg\mathcal{P}\left\{ r_{\eps}\right\} $
for all $\eps\in L$ and for some $L\subzero I$. As usual, we assume
that all the results we proved for $\rhoext$ can also be similarly
proved for the restriction $\rhoext|_{L}$. From \ref{enu:radConvInfty}
for $\rhoext|_{L}$, we have $r<_{L}+\infty$ or $r=_{K}+\infty$
for some $K\subzero L$. The second case implies $r>_{L}\diff\rho^{-Q}$
for all $Q\in\N$. Since $K\subzero I$, we can apply the second alternative
in the implication \eqref{eq:implication} to get $\forall^{0}\eps\in K:\ \mathcal{P}\left\{ r_{\eps}\right\} $,
which gives a contradiction because $K\subseteq L$. We can hence
consider the first case $r<_{L}+\infty$ and apply the subcase \ref{enu:radConvAltNotMod},
i.e.~$r>_{L}\diff\rho^{-P}$ for all $P\in\N$, and we hence proceed
as above applying the second alternative of the implication \eqref{eq:implication}.
In the remaining subcase, we can use \ref{enu:PgeP_m} (with $L$
instead of $I$). This yields $L_{P_{\text{m}}}\subzero L$ and $r\le_{L_{P_{\text{m}}}}\diff\rho^{-P_{\text{m}}}$.
Since $L_{P_{\text{m}}}\subzero I$, we can apply the first alternative
in the implication \eqref{eq:implication} to get once again a contradiction.

\ref{enu:radConvModLe}: Assume that $r>q$ and take $s:=\min(q+1,r)\in\rcrho_{>0}$.
\end{proof}
\noindent Explicitly note the meaning of Lem.~\ref{thm:radConvTool}.\ref{enu:radConvDichotomy}:
on an arbitrary subpoint $r|_{L}$ of the radius of convergence $r=\radconv{a_{n}}$,
we have to consider only two cases: either $r|_{L}$ is $\rho$-moderate
or it is greater than any power $\diff\rho^{-Q}$ (the latter case
including also the case $r|_{L}=+\infty$); if in both cases we are
able to prove the property $\mathcal{P}\left\{ r_{\eps}\right\} $
for $\eps\in L$ sufficiently small, then this property holds for
\emph{all} $\eps$ sufficiently small.

\subsection{Set of convergence}

Even if the radius of convergence of the exponential hyperseries is
$\radconv{\frac{1}{n!}}=+\infty$, we have that $e^{x}=\hypersum{\rho}{\sigma}\frac{x^{n}}{n!}\in\rcrho$
implies $|x|\le\log\left(\diff\rho^{-R}\right)$ for some $R\in\N$:
in other words, the constraint to get a $\rho$-moderate number implies
that even if $\hypersum{\rho}{\sigma}\frac{x^{n}}{n!}$ converges
at $x$, the exponential HPS does \emph{not} converge in the interval
$[x,\radconv{\frac{1}{n!}})=[x,+\infty)\subseteq\rcrho$.

Moreover, in all our examples, if the HPS $\hypersum{\rho}{\sigma}a_{n}(x-c)^{n}\in\rcrho$
converges, then it converges exactly to $\left[\sum_{n=0}^{+\infty}a_{n\eps}(x_{\eps}-c_{\eps})^{n}\right]\in\rcrho$.
The following definition of \emph{set of convergence} closely recalls
the definition of GSF:
\begin{defn}
\label{def:setOfConv}Let $\left(a_{n}\right)_{\text{{\rm c}}}\in\rcrhoc$
and $c\in\rcrho$. The set of convergence
\[
\setconv{a_{n}}{c}
\]
is the set of all $x\in\rcrho$ satisfying
\begin{enumerate}
\item \label{enu:ConvRadConv}$\left|x-c\right|<\radconv{a_{n}}$,
\end{enumerate}
and such that there exist representatives $[x_{\eps}]=x$, $\left[a_{n\eps}\right]_{\text{{\rm c}}}=\left(a_{n}\right)_{\text{{\rm c}}}$
and $[c_{\eps}]=c$ satisfying the following conditions:
\begin{enumerate}[resume]
\item \label{enu:ConvFHPS}$\left[a_{n\eps}\cdot\left(x_{\eps}-c_{\eps}\right)^{n}\right]_{\text{{\rm s}}}\in\hps{x-c}$,
i.e.~we have a formal HPS;
\item \label{enu:ConvConv}$\hypersum{\rho}{\sigma}a_{n}(x-c)^{n}=\left[\sum_{n=0}^{+\infty}a_{n\eps}(x_{\eps}-c_{\eps})^{n}\right]\in\rcrho$;
\item \label{enu:ConvDerMod}For all representatives $[\bar{x}_{\eps}]=x$
and all $k\in\N_{>0}$, the $k$-th derivative net is $\rho$-moderate:
\[
\left(\frac{\diff{}^{k}}{\diff{x}^{k}}\left(\sum_{n=0}^{+\infty}a_{n\eps}(x-c_{\eps})^{n}\right)_{x=\bar{x}_{\eps}}\right)\in\R_{\rho}.
\]
\end{enumerate}
\end{defn}

\noindent Note that condition \ref{enu:ConvFHPS} is necessary because
in \ref{enu:ConvConv} we use a HPS; on the other hand, conditions
\ref{enu:ConvConv} and \ref{enu:ConvDerMod} state that the function
\[
x\in\setconv{a_{n}}{c}\mapsto\hypersum{\rho}{\sigma}a_{n}(x-c)^{n}\in\rti
\]
is a GSF defined by the net of smooth functions $\left(\sum_{n=0}^{+\infty}a_{n\eps}(x_{\eps}-c_{\eps})^{n}\right)$.
As for GSF, see \cite[Thm.~16]{TI}, condition \ref{enu:ConvDerMod}
will be useful to prove that we have independence from representatives
of $x$ in all the derivatives. In Cor.~\ref{cor:setConv_sigma<=00003Drho*},
we will see that under very general assumptions and if $\sigma\le\rho^{*}$,
condition \ref{enu:ConvDerMod} can be omitted.

In Sec.~\ref{exa:Exponential} we will show that $\log\left(\diff\rho^{-1}\right)\in\setconv{\frac{1}{n!}}{0}$
(the set of convergence of the exponential HPS at the origin), but
$\diff\rho^{-1}\notin\text{conv}\left(\left(\frac{1}{n!}\right)_{n}^{\text{{\rm c}}},0\right)$.
We immediately note that $x\in\text{conv}\left(\left(a_{n}\right)_{n}^{\text{c}},c\right)$
if and only if $x-c\in\text{conv}\left(\left(a_{n}\right)_{n}^{\text{c}},0\right)$,
and because of this property without loss of generality we will frequently
assume $c=0$.

We also note that condition \ref{enu:ConvConv} states that the hyperseries
$\hypersum{\rho}{\sigma}a_{n}(x-c)^{n}$ converges, and it does exactly
to the generalized number $\left[\sum_{n=0}^{+\infty}a_{n\eps}(x_{\eps}-c_{\eps})^{n}\right]$.
It is hence natural to wonder whether it is possible that it converges
to some different quantity. This is the problem of the relation between
hyperlimit and $\eps$-wise limit:
\[
\hyperlimarg{\rho}{\sigma}{N}\left[\sum_{n=0}^{\nint{(N)}_{\eps}}a_{n\eps}(x_{\eps}-c_{\eps})^{n}\right],\ \lim_{N\to+\infty}\sum_{n=0}^{N}a_{n\eps}(x_{\eps}-c_{\eps})^{n},
\]
which has been already addressed in \cite[Thm.~12, Thm.~13]{TiGi}.
Intuitively speaking, if the gauge $(\sigma_{\eps})$ is not sufficiently
small, and hence the infinite nets $(\sigma_{\eps}^{-N})$ are not
sufficiently large, it can happen that $\nint{(N)}_{\eps}\to+\infty$
as $\eps\to0$ only very slowly, whereas the $\eps$-wise limit could
require $N\to+\infty$ at a greater speed to converge. This can be
stated more precisely in the following way: Let $\left[a_{n\eps}\cdot\left(x_{\eps}-c_{\eps}\right)^{n}\right]_{\text{{\rm s}}}\in\hps{x-c}$
be a formal HPS and assume that $\sum_{n=0}^{+\infty}a_{n\eps}(x_{\eps}-c_{\eps})^{n}<+\infty$
for $\eps$ small. Then, for all $q\in\N$ and for all $\eps$ small,
we can find $N_{\eps}^{q}\in\N$ such that
\[
\left|\sum_{n=0}^{N_{\eps}^{q}}a_{n\eps}(x_{\eps}-c_{\eps})^{n}-\sum_{n=0}^{+\infty}a_{n\eps}(x_{\eps}-c_{\eps})^{n}\right|\le\rho_{\eps}^{q}\quad\forall n\in\N_{\ge N_{\eps}^{q}}.
\]
However, only if $\left(\sum_{n=0}^{+\infty}a_{n\eps}(x_{\eps}-c_{\eps})^{n}\right)\in\R_{\rho}$
and $(N_{\eps}^{q})\in\R_{\sigma}$, i.e.~$\left[N_{\eps}^{q}\right]\in\hypNs$,
then this also implies $\hypersum{\rho}{\sigma}a_{n}(x-c)^{n}=\left[\sum_{n=0}^{+\infty}a_{n\eps}(x_{\eps}-c_{\eps})^{n}\right]$.

As expected, for HPS the set of convergence is never a singleton:
\begin{thm}
\label{thm:setConvNotTrivial}Let $\left(a_{n}\right)_{\text{{\rm c}}}\in\rcrhoc$
and $c\in\rcrho$. Then
\begin{equation}
\exists q\in\N:\ (c-\diff\rho^{q},c+\diff\rho^{q})\subseteq\setconv{a_{n}}{c}.\label{eq:setConvNotTrivial}
\end{equation}
\end{thm}

\begin{proof}
From Thm.~\ref{thm:radConvTool}.\ref{enu:radConvZero}, we have
$r:=\radconv{a_{n}}\ge\diff\rho^{q_{1}}$ for some $q_{1}\in\N$.
We also have $|a_{n\eps}|\le\rho_{\eps}^{-nQ-R}$ from \eqref{eq:weakMod}.
Assume that $|x-c|<\diff\rho^{q}$: we want to find $q\in\N_{\ge q_{1}}$
so that $x\in\setconv{a_{n}}{c}$. To prove property Def.~\ref{def:setOfConv}.\ref{enu:ConvFHPS},
for $N_{\eps}$, $M_{\eps}\in\N$ and for $\eps$ small, we estimate
\[
\left|\sum_{n=N_{\eps}}^{M_{\eps}}a_{n\eps}(x_{\eps}-c_{\eps})^{n}\right|\le\sum_{n=N_{\eps}}^{M_{\eps}}\rho_{\eps}^{-nQ-R}\rho_{\eps}^{nq}=\rho_{\eps}^{-R}\sum_{n=N_{\eps}}^{M_{\eps}}\rho_{\eps}^{-nQ+nq}.
\]
Therefore, taking $q=\max(1+Q,q_{1})$, we get
\[
\left|\sum_{n=N_{\eps}}^{M_{\eps}}a_{n\eps}(x_{\eps}-c_{\eps})^{n}\right|\le\rho_{\eps}^{-R}\sum_{n=N_{\eps}}^{M_{\eps}}\rho_{\eps}^{n}\le\frac{\rho_{\eps}^{-R}}{1-\rho_{\eps}},
\]
and this proves Def.~\ref{def:setOfConv}.\ref{enu:ConvFHPS}. Similarly,
we have
\begin{align*}
\left|\sum_{n=0}^{M_{\eps}}a_{n\eps}(x_{\eps}-c_{\eps})^{n}-\sum_{n=0}^{+\infty}a_{n\eps}(x_{\eps}-c_{\eps})^{n}\right| & \le\sum_{n=M_{\eps}+1}^{+\infty}\rho_{\eps}^{n}\\
 & \le\frac{\rho_{\eps}^{M_{\eps}+1}}{1-\rho_{\eps}}.
\end{align*}
Since $\hyperlimarg{\rho}{\sigma}{M}\diff\rho^{M+1}=0$, this proves
Def.~\ref{def:setOfConv}.\ref{enu:ConvConv}. Finally, for all $k\in\N_{>0}$
and all representatives $[\bar{x}_{\eps}]=x$, we have
\begin{align}
\frac{\diff{}^{k}}{\diff{x}^{k}}\left(\sum_{n=0}^{+\infty}a_{n\eps}(x-c_{\eps})^{n}\right)_{x=\bar{x}_{\eps}} & =\sum_{n=k}^{+\infty}a_{n\eps}(\bar{x}_{\eps}-c_{\eps})^{n-k}\prod_{j=0}^{k-1}(n-j)\label{eq:der_k}\\
 & =k!\sum_{n=k}^{+\infty}a_{n\eps}(\bar{x}_{\eps}-c_{\eps})^{n-k}{n \choose k},
\end{align}
and hence
\begin{align*}
\left|\frac{\diff{}^{k}}{\diff{x}^{k}}\left(\sum_{n=0}^{+\infty}a_{n\eps}(\bar{x}_{\eps}-c_{\eps})^{n}\right)\right| & \le\sum_{n=k}^{+\infty}\rho_{\eps}^{-nQ-R}\rho_{\eps}^{(n-k)q}\prod_{j=0}^{k-1}(n-j)\\
 & =\rho_{\eps}^{-R-kQ}\sum_{n=k}^{+\infty}\rho_{\eps}^{(q-Q)(n-k)}\prod_{j=0}^{k-1}(n-j)\\
 & =\rho_{\eps}^{-R-kQ}\frac{k!}{(1-\rho_{\eps}^{q-Q})^{k+1}}\in\R_{\rho}.
\end{align*}
In the last step we used $q\ge Q+1$ and the binomial series $\sum_{n=k}^{+\infty}y^{n-k}\prod_{j=0}^{k-1}(n-j)=k!\sum_{n=k}^{+\infty}{n \choose k}y^{n-k}=\frac{k!}{(1-y)^{k+1}}$
for $|y|<1$.
\end{proof}
We can now prove independence from representatives both in Def.~\ref{def:setOfConv}
and in Def.~\ref{def:formalHps}:
\begin{lem}
\label{lem:indepRepr}Let $\left(a_{n}\right)_{\text{{\rm c}}}=\left[a_{n\eps}\right]_{\text{{\rm c}}}=\left[\bar{a}_{n\eps}\right]_{\text{{\rm c}}}\in\rcrhoc$,
$x=[x_{\eps}]=[\bar{x}_{\eps}]$, $c=[c_{\eps}]=[\bar{c}_{\eps}]\in\rcrho$.
Assume that $x\in\setconv{a_{n}}{c}$. Then
\begin{enumerate}
\item \label{enu:series}The nets $\left(a_{n\eps}\right)_{n,\eps}$, $(x_{\eps})$
and $(c_{\eps})$ also satisfy all the conditions of Def.~\ref{def:setOfConv}
of set of convergence.
\item \label{enu:indepRepr}$\left[a_{n\eps}\cdot\left(x_{\eps}-c_{\eps}\right)^{n}\right]_{\text{{\rm s}}}=\left[\bar{a}_{n\eps}\cdot\left(\bar{x}_{\eps}-\bar{c}_{\eps}\right)^{n}\right]_{\text{{\rm s}}}$,
where the equality is in $\rcrhos$.
\end{enumerate}
\end{lem}

\begin{proof}
\ref{enu:series}: Since we have similar steps for several claims,
let $N_{\eps}\in\N$ and $M_{\eps}\in\N\cup\{+\infty\}$, so that
a term of the form $\sum_{n=N_{\eps}}^{M_{\eps}}b_{n\eps}$ represents
both the ordinary series $\sum_{n=0}^{+\infty}b_{n\eps}$ or the finite
sum $\sum_{n=N_{\eps}}^{M_{\eps}}b_{n\eps}$. From Def.~\ref{def:setOfConv}
of set of convergence, we get the existence of representatives $[\hat{x}_{\eps}]=x\in\rcrho$,
$\left[\hat{a}_{n\eps}\right]_{\text{{\rm c}}}=\left(a_{n}\right)_{n}^{\text{{\rm c}}}$
and $[\hat{c}_{\eps}]=c$ satisfying Def.~\ref{def:setOfConv}. Set
$\hat{y}_{\eps}:=\hat{x}_{\eps}-\hat{c}_{\eps}$, $\hat{y}:=[\hat{y}_{\eps}]$.
Let $r:=[r_{\eps}]:=\radconv{a_{n}}$ be the radius of convergence.
From Lem.~\ref{thm:radConvTool}.\ref{enu:radConvModLe}, take $s\in\rcrho$
satisfying $|\hat{y}|<s\le r$ and a representative $[s_{\eps}]=s$
such that $|\hat{y}_{\eps}|<s_{\eps}\le r_{\eps}$ for all $\eps$
small. Set $z_{n\eps}:=a_{n\eps}-\hat{a}_{n\eps}$ and $\hat{z}_{\eps}:=y_{\eps}-\hat{y}_{\eps}$.
For all $k\in\N$, we have
\begin{align}
\sum_{n=N_{\eps}}^{M_{\eps}}a_{n\eps}y_{\eps}^{n-k}{n \choose k} & =\sum_{n=N_{\eps}}^{M_{\eps}}\left(\hat{a}_{n\eps}+z_{n\eps}\right)\left(\hat{y}_{\eps}+\hat{z}_{\eps}\right)^{n-k}{n \choose k}\nonumber \\
 & =\sum_{n=N_{\eps}}^{M_{\eps}}\hat{a}_{n\eps}\left(\hat{y}_{\eps}+\hat{z}_{\eps}\right)^{n-k}{n \choose k}+\sum_{n=N_{\eps}}^{M_{\eps}}z_{n\eps}\left(\hat{y}_{\eps}+\hat{z}_{\eps}\right)^{n-k}{n \choose k}.\label{eq:a_hat_a}
\end{align}
Since $\hat{z}=0$ , we also have $|\hat{y}_{\eps}|+|\hat{z}_{\eps}|<s_{\eps}\le r_{\eps}$
for all $\eps$ small. For the same $\eps$, assume that $|z_{n\eps}|\le\rho_{\eps}^{np+q}$
for fixed arbitrary $p$, $q\in\N$. We first consider the second
summand in \eqref{eq:a_hat_a}:

\begin{multline*}
\left|\sum_{n=N_{\eps}}^{M_{\eps}}z_{n\eps}\left(\hat{y}_{\eps}+\hat{z}_{\eps}\right)^{n-k}{n \choose k}\right|\le\sum_{n=N_{\eps}}^{M_{\eps}}\rho_{\eps}^{np+q}s_{\eps}^{n-k}{n \choose k}=\rho_{\eps}^{q+kp}\sum_{n=N_{\eps}}^{M_{\eps}}\left(\rho_{\eps}^{p}s_{\eps}\right)^{n-k}{n \choose k}\\
\le\rho_{\eps}^{q+kp}\sum_{n=k}^{+\infty}\left(\rho_{\eps}^{p}s_{\eps}\right)^{n-k}{n \choose k}-\rho_{\eps}^{q+kp}\sum_{n=k}^{N_{\eps}-1}\left(\rho_{\eps}^{p}s_{\eps}\right)^{n-k}{n \choose k}\\
\le2\rho_{\eps}^{q+kp}\sum_{n=k}^{+\infty}\left(\rho_{\eps}^{p}s_{\eps}\right)^{n-k}{n \choose k}.
\end{multline*}
Since $s\in\rcrho$, we can take $p\in\N$ sufficiently large so that
$\rho_{\eps}^{p}s_{\eps}<1$. This implies
\[
\left|\sum_{n=N_{\eps}}^{M_{\eps}}z_{n\eps}\left(\hat{y}_{\eps}+\hat{z}_{\eps}\right)^{n-k}{n \choose k}\right|\le\frac{2\rho_{\eps}^{q+kp}}{\left(1-\rho_{\eps}^{p}s_{\eps}\right)^{k+1}}.
\]
Thereby, for $q\to+\infty$, this summand defines a negligible net.
For the first summand of \eqref{eq:a_hat_a}, we can use the mean
value theorem to get
\begin{multline}
\left|\sum_{n=N_{\eps}}^{M_{\eps}}\hat{a}_{n\eps}\left(\hat{y}_{\eps}+\hat{z}_{\eps}\right)^{n-k}{n \choose k}-\sum_{n=N_{\eps}}^{M_{\eps}}\hat{a}_{n\eps}\hat{y}_{\eps}^{n-k}{n \choose k}\right|\\
\le\left|\sum_{n=N_{\eps}}^{M_{\eps}}\hat{a}_{n\eps}(n-k)\xi_{\eps}^{n-k-1}{n \choose k}\hat{z}_{\eps}\right|=\left|\hat{z}_{\eps}\right|\left|\sum_{n=N_{\eps}}^{M_{\eps}}\hat{a}_{n\eps}(n-k)\xi_{\eps}^{n-k-1}{n \choose k}\right|\label{eq:meanV}
\end{multline}
for some $\xi_{\eps}\in[\hat{y}_{\eps},\hat{y}_{\eps}+\hat{z}_{\eps}]\cup[\hat{y}_{\eps}+\hat{z}_{\eps},\hat{y}_{\eps}]$.
Thereby, the right hand side of \eqref{eq:meanV} is negligible because
of Def.~\ref{def:setOfConv}.\ref{enu:ConvDerMod}.

\noindent We can hence state that for all $k\in\N$
\begin{equation}
\left(\sum_{n=N_{\eps}}^{M_{\eps}}a_{n\eps}\left(y_{\eps}+z_{\eps}\right)^{n-k}{n \choose k}\right)\sim_{\rho}\left(\sum_{n=N_{\eps}}^{M_{\eps}}\hat{a}_{n\eps}\hat{y}_{\eps}^{n-k}{n \choose k}\right).\label{eq:a-hat-a_equiv}
\end{equation}
In the case $M_{\eps}<+\infty$ for all $\eps$ and $k=0$, this proves
that $\left[a_{n\eps}\cdot y_{\eps}^{n}\right]_{\text{{\rm s}}}\in\rcrhos$
because $(\hat{a}_{n\eps})_{n,\eps}$ and $(\hat{y}_{\eps})$ satisfy
Def.~\ref{def:setOfConv}.\ref{enu:ConvFHPS}. In the case $M_{\eps}=+\infty$
and $N_{\eps}=0=k$, it also proves the moderateness of $\left(\sum_{n=0}^{+\infty}a_{n\eps}y_{\eps}^{n}\right)$,
i.e.~the implicit moderateness requirement of Def.~\ref{def:setOfConv}.\ref{enu:ConvConv}.
Finally, for $k>0$, property \eqref{eq:a-hat-a_equiv} also shows
that Def.~\ref{def:setOfConv}.\ref{enu:ConvDerMod} also holds for
$(a_{n\eps})$ and $(y_{\eps})$ because of \eqref{eq:der_k}. We
can also apply \eqref{eq:a-hat-a_equiv} with $k=0$ to $\left[\bar{a}_{n\eps}\right]_{\text{{\rm c}}}=\left(a_{n}\right)_{\text{{\rm c}}}$,
$[\bar{x}_{\eps}]=x$, $[\bar{c}_{\eps}]=c$, $\bar{y}_{\eps}:=\bar{x}_{\eps}-\bar{c}_{\eps}$
and with $M_{\eps}<+\infty$, to get
\[
\left(\sum_{n=N_{\eps}}^{M_{\eps}}a_{n\eps}y_{\eps}^{n}\right)\sim_{\rho}\left(\sum_{n=N_{\eps}}^{M_{\eps}}\hat{a}_{n\eps}\hat{y}_{\eps}^{n}\right)\sim_{\rho}\left(\sum_{n=N_{\eps}}^{M_{\eps}}\bar{a}_{n\eps}\bar{y}_{\eps}^{n}\right).
\]
This proves claim \ref{enu:indepRepr} and hence also Def.~\ref{def:setOfConv}.\ref{enu:ConvConv}
because $\hypersum{\rho}{\sigma}\left[\hat{a}_{n\eps}\right]\cdot\left[\hat{y}_{\eps}\right]^{n}$
converges to $\left[\sum_{n=0}^{+\infty}\hat{a}_{n\eps}\hat{y}_{\eps}^{n}\right]=\left[\sum_{n=0}^{+\infty}a_{n\eps}y_{\eps}^{n}\right]=\left[\sum_{n=0}^{+\infty}\bar{a}_{n\eps}\bar{y}_{\eps}^{n}\right]\in\rcrho$
from \eqref{eq:a-hat-a_equiv}.
\end{proof}

\subsection{\label{subsec:Examples}Examples}

We start studying geometric hyperseries, which in general are convergent
HPS if $\sigma\le\rho^{*}$:
\begin{example}[Geometric hyperseries]
\label{exa:Geometric-series}Assume that $x\in(-1,1)\subseteq\rti$.
We have:

\begin{equation}
\left[\left|\sum_{n=0}^{N_{\eps}}x_{\eps}^{n}\right|\right]\leq\left[\left|\frac{1-|x_{\eps}^{N_{\eps}+1}|}{1-x_{\eps}}\right|\right]\le\frac{2}{1-x}\in\rcrho.\label{eq:geomModer}
\end{equation}
This shows that $(x^{n})_{n}=[x_{\eps}^{n}]\in\rcrhos$ for all gauges
$\rho$, $\sigma$. Hence by Def.~\ref{def:formalHps}, $[x_{\eps}^{n}]{}_{\text{{\rm s}}}\in\hps{x}$,
i.e.~the geometric series is a formal hyper-series. Since coefficients
$a_{n\eps}=1$, we have, $[a_{n\eps}]_{{\rm c}}\in\rcrhoc$ (see Def.~\ref{def:radCon}.\ref{enu:coeffHPS}).
Now, by Def.~\ref{def:radCon}.\ref{enu:radConv}, $\radconv{1}=1$.
From Def.~\ref{def:setOfConv}.\ref{enu:ConvRadConv}, we have $\setconv{1}{0}\subseteq(-1,1)$.
Now, take $x=[x_{\eps}]\in(-1,1)$, with $-1<x_{\eps}<1$ for all
$\eps$. From \cite[Example~8]{TiGi}, if $\sigma\leq\rho^{*}$ (i.e.~if
$\exists Q\in\R_{>0}\,\forall^{0}\eps:\ \sigma_{\eps}\le\rho_{\eps}^{Q}$),
we have Def.~\ref{def:setOfConv}.\ref{enu:ConvConv}. Finally, if
$[\bar{x}_{\eps}]=x$ is another representative and $k\in\N_{>0}$,
then $-1<\bar{x}_{\eps}<1$ for $\eps$ small, and from \eqref{eq:der_k}
we get $\sum_{n=k}^{+\infty}k!{n \choose k}\bar{x}_{\eps}^{n-k}=\frac{k!}{(1-\bar{x}_{\eps})^{k+1}}\in\R_{\rho}$
because $1-x>0$ is invertible. Note explicitly that $\sigma\le\rho^{*}$
is a sufficient condition ensuring the convergence of \emph{any} geometric
hyperseries with $|x|<1$. However, we already used (see e.g.~Thm.~\ref{thm:setConvNotTrivial})
the convergence of the geometric hyperseries $\hypersum{\rho}{\sigma}\diff\rho^{n}=\frac{1}{1-\diff\rho}$
for all gauges $\rho$, $\sigma$. More generally, exactly as proved
in \cite[Example~8]{TiGi}, it is easy to see that $\hypersum{\rho}{\sigma}x^{n}=\frac{1}{1-x}$
if $\sigma_{\eps}\le\left(\frac{\log x_{\eps}}{\log\rho_{\eps}}\right)^{Q}$
for $\eps$ small and some $Q\in\R_{>0}$.
\end{example}

\begin{example}[A smooth function with a flat point]
\label{exa:flatPoint}Consider the GSF corresponding to the ordinary
smooth function $f(x):=\begin{cases}
e^{-1/x} & \text{if }x\in\R_{>0}\\
0 & \text{otherwise}
\end{cases}$. It is not hard to prove that $|f(x)|\le|x|^{q}$ for all $x\approx0$
and all $q\in\N$. Thereby, $f(x)=0$ for all $x$ such that $|x|\leq\diff\rho^{r}$
for some $r\in\R_{>0}$. Therefore, we trivially have $f(x)=\hypersum{\rho}{\sigma}0\cdot x{}^{n}$
only for all $x$ in this infinitesimal neighborhood of $0$. On the
other hand, $\setconv{0}{0}=\rcrho$. Moreover, $\radconv{\frac{f^{(n)}(c)}{n!}}=+\infty$
and $\setconv{\frac{f^{(n)}(c)}{n!}}{c}=\rti$ for all $c\in\rti$
such that $|c|\gg0$, i.e.~satisfying $|c|\ge r$ for some $r\in\R_{>0}$,
but $f(x)=$$\hypersum{\rho}{\sigma}\frac{f^{(n)}(c)}{n!}\cdot\left(x-c\right)^{n}$
only for all $x\in\rti$ such that $|x|\gg0$, which is a strict subset
of $\setconv{\frac{f^{(n)}(c)}{n!}}{c}=\rti$. The GSF $f$ is therefore
a candidate to be a GRAF, but not an entire GRAF.
\end{example}

\begin{example}[A nowhere analytic smooth function]
\noindent \label{subsec:A-smooth-nowhere}A classical example of
an infinitely differentiable function which is not analytic at any
point is $F(x)=\sum_{k\in2^{\N}}e^{-\sqrt{k}}\cos(kx)$, where $2^{\N}:=\left\{ 2^{n}\mid n\in\N\right\} $.
Since for all $x=\pi\frac{p}{q}$, with $p\in\N$ and $q\in2^{\N}$
and for all $n\in2^{\N}$, $n\ge4$, $n>q$, we have $F^{(n)}(x)\ge e^{-2n}(4n^{2})^{n}+O(q^{n})$
as $n\to+\infty$, we have that $\left(\frac{F^{(n)}(x)}{n!}\right)_{n,\eps}\notin\left(\R^{\N\times I}\right)_{\rho}$,
i.e.~they are \emph{not} coefficients for a HPS.
\end{example}

\begin{example}[Exponential]
\label{exa:Exponential}We clearly have $\left(\frac{1}{n!}\right)_{\text{c}}\in\rcrhoc$
and $\radconv{\frac{1}{n!}}=+\infty$, i.e.~we have coefficients
for an HPS with infinite radius of convergence. Set $C:=\left\{ x\in\rcrho\mid\exists K\in\N:\ |x|<-K\log\diff\rho\right\} $.
For all $x=[x_{\eps}]\in C$ and all $N_{\eps}$, $M_{\eps}\in\N$,
we have $\left|\sum_{n=N_{\eps}}^{M_{\eps}}\frac{x_{\eps}^{n}}{n!}\right|\le e^{|x_{\eps}|}\le\rho_{\eps}^{-K}$
for $\eps$ small, and this shows that $\left[\frac{x^{n}}{n!}\right]_{\text{{\rm s}}}\in\hps{x}$,
i.e.~for all $x\in C$, we have a formal HPS. We finally want to
prove that $C=\setconv{\frac{1}{n!}}{0}$ if $\sigma\le\rho^{*}$.
The inclusion $\supseteq$ follows directly from Def.~\ref{def:setOfConv}.\ref{enu:ConvConv}.
If $x=[x_{\eps}]\in C$, then condition Def.~\ref{def:setOfConv}.\ref{enu:ConvDerMod}
holds because the $k$-th derivative HPS $\left(k!\sum_{n=k}^{+\infty}{n \choose k}\cdot\frac{x_{\eps}^{n-k}}{n!}\right)=\left(e^{x_{\eps}}\right)\in\R_{\rho}$.
To prove Def.~\ref{def:setOfConv}.\ref{enu:ConvConv}, assume that
$|x_{\eps}|<-K\log\rho_{\eps}=:M_{\eps}$ for all $\eps$ and set
$M:=[M_{\eps}]\in\rcrho$. Take $N=[N_{\eps}]\in\hypNs$ such that
$\frac{M}{N+1}<\frac{1}{2}$, so that, exactly as in \cite[Example~8]{TiGi},
we can prove that $\frac{M^{n+1}}{(n+1)!}<\frac{1}{2^{n+1}}$ and
hence $\left|\sum_{n=N_{\eps}+1}^{+\infty}\frac{x_{\eps}^{n}}{n!}\right|\le\sum_{n\ge N_{\eps}}\frac{1}{2^{n}}\to0$
as $N\to+\infty$, $N\in\hypNs$, if $\sigma\le\rho^{*}$. Similarly,
we can consider trigonometric functions whose set of convergence is
the entire $\rti$.
\end{example}

\begin{example}[\label{exa:Dirac-delta}Dirac delta]
\label{subsec:Dirac-delta}In Rem.~\ref{rem:radConv}.\ref{enu:deltaRadConv},
we already proved that $\left(\frac{\delta^{(n)}(0)}{n!}\right)_{\text{c}}\in\rcrhoc$
and $\radconv{\frac{\delta^{(n)}(0)}{n!}}=+\infty$. For all $x=[x_{\eps}]\in\rti$
and all $N_{\eps}$, $M_{\eps}\in\N,$ we have $\left|\sum_{n=N_{\eps}}^{M_{\eps}}\frac{\delta_{\eps}^{(n)}(0)}{n!}x_{\eps}^{n}\right|\le b_{\eps}\cdot\sum_{n=0}^{+\infty}\frac{|\mu^{(n)}(0)|}{n!}|b_{\eps}x_{\eps}|^{n}$.
But $|\mu^{(n)}(0)|=i^{n}\mu^{(n)}(0)$ because $\mu^{(n)}(0)=0$
if $n$ is odd, so that $\left|\sum_{n=N_{\eps}}^{M_{\eps}}\frac{\delta_{\eps}^{(n)}(0)}{n!}x_{\eps}^{n}\right|\le b_{\eps}\cdot\sum_{n=0}^{+\infty}\frac{\mu^{(n)}(0)}{n!}|ib_{\eps}x_{\eps}|^{n}=b_{\eps}\mu(i|b_{\eps}x_{\eps}|)\in\R_{\rho}$,
and this proves that $\left(\frac{\delta^{(n)}(0)}{n!}\right)_{\text{c}}\in\hps{x}$,
i.e.~we always have a formal HPS. Condition Def.~\ref{def:setOfConv}.\ref{enu:ConvDerMod}
follows because derivatives $\delta^{(k)}(x)\in\rti$ are always moderate.
It remains to prove Def.~\ref{def:setOfConv}.\ref{enu:ConvConv}
for all $x=[x_{\eps}]\in\rti$ to show that $\setconv{\frac{\delta^{(n)}(0)}{n!}}{0}=\rti$:
\[
\sum_{n=0}^{N}\frac{\delta^{(n)}(0)}{n!}x^{n}=\left[b_{\eps}\sum_{n=0}^{N_{\eps}}\frac{\mu^{(n)}(0)}{n!}b_{\eps}^{n}x_{\eps}^{n}\right]=\delta(x)-b\left[\mu^{(N_{\eps}+1)}(\bar{x}_{\eps})\frac{x_{\eps}^{N_{\eps}+1}}{(N_{\eps}+1)!}\right]
\]
where the existence of $\bar{x}_{\eps}\in[0,x_{\eps}]\cup[x_{\eps},0]$
is derived from Taylor's formula. Since $|\mu^{(k)}(y)|\le\frac{1}{2\pi}\int\beta(x)|x|^{k}\,\diff x=:C\in\R_{>0}$
for all $k\in\N$ and all $y\in\R$, we obtain
\[
\left|\sum_{n=0}^{N}\frac{\delta^{(n)}(0)}{n!}x^{n}-\delta(x)\right|\le bC\left[\frac{|x_{\eps}|^{N_{\eps}+1}}{(N_{\eps}+1)!}\right].
\]
Using Stirling's approximation, we have $\frac{|x_{\eps}|^{N_{\eps}+1}}{(N_{\eps}+1)!}\le2\left(\frac{|x_{\eps}|e}{N_{\eps}}\right)^{N}\le\rho_{\eps}^{N_{\eps}}$
for all $N\in\hypNs$ such that $N>|x|e\diff\rho^{-1}$, which is
always possible if $\sigma\le\rho^{*}$. Since $\hyperlimarg{\rho}{\sigma}{N}\diff\rho^{N}=0$,
this proves the claim.
\end{example}

\noindent A different way to include a large class of examples is
to use the characterization Thm.~\ref{thm:growthDers} by factorial
growth of derivatives of GRAF.

When we say that a HPS $\hypersum{\rho}{\sigma}a_{n}(x-c)^{n}$ is
\emph{convergent}, we already assume that its coefficients are correctly
chosen and that the point $x$ is in the set of convergence, as stated
in the following
\begin{defn}
\label{def:convHps} We say that $\hypersum{\rho}{\sigma}a_{n}(x-c)^{n}$
is a \emph{convergent} HPS if
\begin{enumerate}
\item $(a_{n})_{c}\in\rcrhoc$ are coefficients for HPS.
\item $x\in\setconv{a_{n}}{c}$.
\end{enumerate}
\end{defn}

In all the previous examples, we recognized that dealing with HPS
is more involved than working with ordinary series, where we only
have to check that the final result is a convergent series ``of the
form'' $\sum_{n=0}^{\infty}a_{n}(x-c)^{n}$. On the contrary, for
HPS we have to control the following steps:
\begin{enumerate}[label=\arabic*)]
\item We have to check that the net $(a_{n\eps})_{n,\eps}$ defines coefficients
for HPS (Def.~\ref{def:radCon}.\ref{enu:coeffHPS}), i.e.~that
\[
\forall^{0}\eps\,\forall n\in\N:\ |a_{n\eps}|\le\rho_{\eps}^{-nQ-R}
\]
for some $Q$, $R\in\N_{>0}$. This allows us to talk of the radius
of convergence $\radconv{a_{n}}$ and of the set of convergence $\setconv{a_{n}}{c}$
(Def.~\ref{def:setOfConv}). Because of Thm.~\ref{thm:setConvNotTrivial},
this set is always non-trivial
\begin{equation}
(c-\diff\rho^{q},c+\diff\rho^{q})\subseteq\setconv{a_{n}}{c}\subseteq(c-\radconv{a_{n}},c+\radconv{a_{n}}),\label{eq:boundsSetConv}
\end{equation}
but in general \emph{is not an interval}, like the case of the exponential
function clearly shows. This step already allows us to say that the
HPS $\hypersum{\rho}{\sigma}a_{n}(x-c)^{n}$ is convergent, i.e.~Def.~\ref{def:convHps},
if $x\in\setconv{a_{n}}{c}$.
\item At this point, we can study the set of convergence, e.g.~to arrive
at an explicit form $C=\setconv{a_{n}}{c}\subseteq(c-\radconv{a_{n}},c+\radconv{a_{n}})$.
This depends mainly on three conditions:
\begin{enumerate}[label=\alph*)]
\item For all $x\in C$, we must have a formal HPS (Def.~\ref{def:formalHps})
because this allows us to talk of any hyperfinite sum $\sum_{n=N}^{M}a_{n}(x-c)^{n}$
for $M$, $N\in\hypNs$. Here, the main step is to prove that the
net $\left(\sum_{n=N_{\eps}}^{M_{\eps}}a_{n\eps}(x_{\eps}-c_{\eps})^{n}\right)\in\R_{\rho}$.
\item For all $x\in C$, we have to check Def.~\ref{def:setOfConv}.\ref{enu:ConvConv},
i.e.~the equality:
\begin{equation}
\hypersum{\rho}{\sigma}a_{n}(x-c)^{n}=\left[\sum_{n=0}^{+\infty}a_{n\eps}(x_{\eps}-c_{\eps})^{n}\right]\in\rcrho.\label{eq:1cond}
\end{equation}
\item Finally, we have to prove that for all representatives $x=[\bar{x}_{\eps}]\in C$,
all the derivatives $\frac{\diff{}^{k}}{\diff{x}^{k}}\left(\sum_{n=0}^{+\infty}a_{n\eps}(\bar{x}_{\eps}-c_{\eps})^{n}\right)$
are $\rho$-moderate.
\item After the previous three steps, we get $C\subseteq\setconv{a_{n}}{c}$,
and hence it remains to prove the opposite inclusion.
\end{enumerate}
See Cor.~\ref{cor:setConv_sigma<=00003Drho*} for sufficiently general
conditions under which only \eqref{eq:1cond} suffices to prove that
$x$ lies in the set of convergence.
\end{enumerate}
\noindent Note explicitly that we \emph{never} formally defined what
is a HPS: we have \emph{formal HPS} (Def.~\ref{def:formalHps}),
the notion of \emph{coefficients for HPS} (Def.~\ref{def:radCon}.\ref{enu:coeffHPS}),
which always have a strictly positive \emph{radius of convergence}
(Def.~\ref{def:radCon}.\ref{enu:radConv}) and a non trivial \emph{set
of convergence} (Def.~\ref{def:setOfConv} and Thm.~\ref{thm:setConvNotTrivial}),
and finally \emph{convergent HPS }(Def.~\ref{def:convHps}).

\subsection{Topological properties of the set of convergence}

The first consequence of our definition of convergent HPS Def.~\ref{def:convHps}
and radius of convergence Def.~\ref{def:radCon}, is the following
\begin{lem}
\label{lem:sup_eq_rad}Let $\hypersum{\rho}{\sigma}a_{n}(x-c)^{n}$
be a convergent HPS. If the following least upper bound exists
\begin{equation}
\text{\emph{lub}}\left\{ |\bar{x}-c|\mid{\textstyle \hypersum{\rho}{\sigma}a_{n}(\bar{x}-c)^{n}}\text{ is a convergent HPS}\right\} =:r\in\rti,\label{eq:lub}
\end{equation}
then $r\le\radconv{a_{n}}$.
\end{lem}

\begin{proof}
In fact, if $\hypersum{\rho}{\sigma}a_{n}(\bar{x}-c)^{n}$ is a convergent
HPS, then $\hypersum{\rho}{\sigma}a_{n}(\bar{x}-c)^{n}=\left[\sum_{n=0}^{+\infty}a_{n\eps}(\bar{x}_{\eps}-c_{\eps})^{n}\right]$,
and hence $|\bar{x}_{\eps}-c_{\eps}|\leq(\limsup_{n}\left|a_{n\eps}\right|^{1/n})^{-1}$
for all $\eps$ small, i.e.~$|\bar{x}-c|\le\radconv{a_{n}}$.
\end{proof}
\noindent Note that from Example \ref{exa:Exponential}, we have that
the least upper bound of
\begin{equation}
\left\{ x\in\rcrho\mid{\textstyle \hypersum{\rho}{\sigma}\frac{x^{n}}{n!}}\text{ is a convergent HPS}\right\} \label{eq:expConv}
\end{equation}
does not exist in $\rti$, whereas Def.~\ref{def:radCon} yields
the value $\radconv{\frac{1}{n!}}=+\infty$. Therefore, Def.~\ref{def:radCon}
allows us to consider the exponential HPS even if the supremum of
\eqref{eq:expConv} does not exist. It remains an open problem whether
$r=\radconv{a_{n}}$, at least if the least upper bound \eqref{eq:lub},
or the corresponding sharp supremum, exists.

We now study absolute convergence of HPS, and sharply boundedness
of the summands of a HPS. We first show that the hypersequence $\left(a_{n}(x-c)^{n}\right)_{n\in\hypNs}$
of the terms of a HPS is sharply bounded:
\begin{lem}
\label{lem:kboundhypersum}Let $x$, $c\in\rti$. If $\hypersum{\rho}{\sigma}a_{n}(x-c)^{n}$
is a convergent HPS, then 
\begin{equation}
\text{\ensuremath{\exists K\in\rcrho\,\forall n\in\hypNs:\ |a_{n}(x-c)^{n}|<K}}.\label{eq:KboundHPS}
\end{equation}
We recall that because of the definition of formal HPS (Def.~\ref{def:formalHps})
and \cite[Lem.~7]{TiGi} the term $a_{n}(x-c)^{n}\in\rti$ is well-defined
for all $n\in\hypNs$.
\end{lem}

\begin{proof}
Set $\bar{x}:=x-c$, i.e.~without loss of generality we can assume
$c=0$. Since $\hypersum{\rho}{\sigma}a_{n}\bar{x}^{n}$ converges,
from \cite[Lem.~15]{TiGi} we have

\begin{equation}
\exists N\in\hypNs\,\forall n\in\hypNs_{\geq N}:\ |a_{n}\bar{x}^{n}|<1.\label{eq:1boundHPS}
\end{equation}

\noindent Let us consider an arbitrary $n\in\hypNs$. From \cite[Lem.~13]{MTAG},
we have either $n\geq N$ or $n<_{L}$$N$ for some $L\subzero I$.
In the latter case, $|a_{n}\bar{x}^{n}|\le_{L}s:=\sum_{n=0}^{N-1}|a_{n}\bar{x}^{n}|<\max(s+1,1)=:K$.
From \cite[Lem. 7.(iii)]{MTAG} and from \eqref{eq:1boundHPS}, the
claim follows.
\end{proof}
\noindent The previous proof is essentially the generalization in
our setting of the classical one, see e.g.~\cite{KrPa02}. However,
property \eqref{eq:KboundHPS} does not allow us to apply the direct
comparison test \cite[Thm.~22]{TiGi}. Indeed, let us imagine that
we only prove $\left|a_{n}x^{n}\right|<Kh^{n}$, with $h<1$, for
all $n\in\hypNs$ and with $K$ coming from \eqref{eq:KboundHPS};
as we already explained in \cite[Sec.~3.3]{TiGi}, this would imply
\[
\forall n\in\N\,\exists\eps_{0n}\,\forall\eps\le\eps_{0n}:\ \left|a_{n\eps}x_{\eps}^{n}\right|\le K_{\eps}h_{\eps}^{n},
\]
and the dependence of $\eps_{0n}$ from $n\in\N$ is a problem in
estimating inequalities of the form $\sum_{n=0}^{N_{\eps}}\left|a_{n\eps}x_{\eps}^{n}\right|\le K_{\eps}\sum_{n=0}^{N_{\eps}}h_{\eps}^{n}$,
see \cite{TiGi}. A solution of this problem is to consider a uniform
property of $n\in\N$ with respect to $\eps$:
\begin{defn}
\label{def:eventInf}Let $(a_{n})_{{\rm c}}\in\rcrhoc$ and $x$,
$c\in\rti$, then we say that $\left(a_{n}(x-c)^{n}\right)_{n\in\N}$
\emph{is eventually $\rti$-bounded in} $\rcrhoc$, if there exist
representatives $(a_{n})_{{\rm c}}=\left[a_{n\eps}\right]_{\text{{\rm c}}}$,
$[x_{\eps}]=x$, $[c_{\eps}]=c$ such that
\begin{equation}
\exists[R_{\eps}]\in\rti\,\exists N\in\N\,\forall^{0}\eps\,\forall n\in\N_{\ge N}:\ \left|a_{n\eps}(x_{\eps}-c_{\eps})^{n}\right|<R_{\eps}.\label{eq:eventBound}
\end{equation}
\end{defn}

\begin{rem}
\label{rem:eventBound}\ 
\begin{enumerate}
\item The adverb \emph{eventually} clearly refers to the validity of the
uniform inequality in \eqref{eq:eventBound} only for $n$ sufficiently
large.
\item If for $\eps$ small, the series $\sum_{n=0}^{+\infty}\left|a_{n\eps}(x_{\eps}-c_{\eps})^{n}\right|=:R_{\eps}$
of absolute values terms converges to a $\rho$-moderate net, then
\eqref{eq:eventBound} holds for $N=0$. This includes Example \ref{exa:Geometric-series}
of geometric hyperseries, Example \ref{exa:flatPoint} of a function
with a flat point if both $x$, $c$ are finite, and Example \ref{exa:Exponential}
of the exponential hyperseries at $c=0$ if $x$ is finite.
\item \label{enu:eventBoundDelta}In Example \ref{exa:Dirac-delta} of Dirac
delta at $c=0$, if $|bx|\le1$ (therefore, $x$ is an infinitesimal
number) we have $\left|\frac{\delta_{\eps}^{(n)}(0)}{n!}x_{\eps}^{n}\right|=\left|\frac{\mu^{(n)}(0)}{n!}b_{\eps}^{n+1}x_{\eps}^{n}\right|\le b_{\eps}$
for all $n\in\N$ such that $\left|\frac{\mu^{(n)}(0)}{n!}\right|\le1$.
Therefore, $\left(\frac{\delta^{(n)}(0)}{n!}x^{n}\right)_{n\in\N}$
is eventually $\rti$-bounded in $\rcrhoc$ if $|bx|\le1$. If $x\gg0$,
i.e.~$x\ge s\in\R_{>0}$, then $\left|\frac{\delta_{\eps}^{(n)}(0)}{n!}x_{\eps}^{n}\right|=\left|\frac{\mu^{(n)}(0)}{n!}b_{\eps}^{n+1}x_{\eps}^{n}\right|\ge\left|\frac{\mu^{(n)}(0)}{n!}\right|s^{n}b_{\eps}^{n+1}$
and hence condition \eqref{eq:eventBound} does not hold for any $[R_{\eps}]\in\rti$
because $b\ge\diff\rho^{-a}$ for some $a\in\R_{>0}$ (see e.g.~\cite[Sec.~3.0.2]{TI}).
\end{enumerate}
\end{rem}

\noindent The last example also shows that property \eqref{eq:eventBound}
does not hold for all point $x\in\setconv{a_{n}}{c}$. However, it
always holds for any $c$ if $x$ is sufficiently near to $c$:
\begin{lem}
\label{lem:eventBoundNear_c}Let $(a_{n})_{{\rm c}}\in\rcrhoc$ and
$c\in\rti$, then there exists $\sigma\in\R_{>0}$ such that for all
$x\in B_{\sigma}(c)$, the sequence of summands $\left(a_{n}(x-c)^{n}\right)_{n\in\N}$
is eventually $\rti$-bounded in $\rcrhoc$.
\end{lem}

\begin{proof}
Using the same notation as above, since $(a_{n})_{{\rm c}}\in\rcrhoc$,
we have $\forall^{0}\eps\,\forall n\in\N:\ |a_{n\eps}|\le\rho_{\eps}^{-nQ-R}$.
Therefore, for $\sigma:=\diff\rho^{Q}$, we have $|a_{n\eps}(x_{\eps}-c_{\eps})^{n}|\le\rho_{\eps}^{-nQ-R}\rho_{\eps}^{nQ}=\rho_{\eps}^{-R}$.
\end{proof}
The following result is a stronger version of the previous Lem.~\ref{lem:kboundhypersum},
and allow us to apply the dominated convergence test:
\begin{lem}
\label{lem:KstrongBound}Let $(a_{n})_{{\rm c}}\in\rcrhoc$, $x$,
$c\in\rti$, and assume that $\left(a_{n}(x-c)^{n}\right)_{n\in\N}$
\emph{is eventually $\rti$-bounded in} $\rcrhoc$, then
\begin{equation}
\exists K\in\rcrho:\ \left(\left(a_{n}(x-c)^{n}\right)\right)_{\text{\emph{c}}}<K\text{ in }\rcrhoc,\label{eq:eventBoundIntr}
\end{equation}
i.e.~for all representatives $(a_{n})_{{\rm c}}=\left[a_{n\eps}\right]_{\text{{\rm c}}}$,
$[x_{\eps}]=x$, $[c_{\eps}]=c$, $[K_{\eps}]=K$, we have
\begin{equation}
\forall^{0}\eps\,\forall n\in\N:\ \left|a_{n\eps}(x_{\eps}-c_{\eps})^{n}\right|<K_{\eps}.\label{eq:eventBoundEps}
\end{equation}
Since $\rti\subseteq\rcrhoc$ by Rem.~\ref{rem:radConv}.\ref{enu:rcrhoSubsetRhoext},
property \eqref{eq:eventBoundIntr} also shows that Def.~\ref{def:eventInf}
does not depend on the representatives involved.
\end{lem}

\begin{proof}
It suffices to set $K:=R\vee\max_{n\le N}a_{n}$, where $R\in\rti$
and $N\in\N$ come from \eqref{eq:eventBound}.
\end{proof}
Even if the case of the exponential HPS (see Example \ref{exa:Exponential})
shows that in general the set of convergence is not an interval, it
has very similar properties, at least if the gauge $\sigma$ is sufficiently
small:
\begin{thm}
\label{thm:absConvTopProp}Let $\sigma\leq\rho^{*}$ and $\hypersum{\rho}{\sigma}a_{n}(\bar{x}-c)^{n}$
be a convergent HPS whose sequence of summands $\left(a_{n}(\bar{x}-c)^{n}\right)_{n\in\N}$
is eventually $\rti$-bounded in $\rcrhoc$. Then for all $x\in B_{|\bar{x}-c|}(0)$
we have:
\begin{enumerate}
\item \label{enu:absConvUnif}The HPS converges absolutely at $x$, and
hence uniformly on every functionally compact $K\fcmp\overline{B}_{|\bar{x}-c|}(c)$;
\item \label{enu:eventuallyBound}$\left(a_{n}(x-c)^{n}\right)_{n\in\N}$
is eventually $\rti$-bounded in $\rcrhoc$;
\item \label{enu:notBord}If $|\hat{x}-c|=|\bar{x}-c|$, then not necessarily
$\hypersum{\rho}{\sigma}a_{n}(\hat{x}-c)^{n}$ converges.
\end{enumerate}
If $\hypersum{\rho}{\sigma}a_{n}(x-c)^{n}=[\sum_{n=0}^{\infty}a_{n\eps}(x_{\eps}-c_{\eps})^{n}]\in\rcrho$,
then:
\begin{enumerate}
\item \label{enu:underConv}$x\in\setconv{a_{n}}{c}$;
\item \label{enu:underInterior}$x$ is a sharply interior point, i.e.~$B_{s}(x)\subseteq\setconv{a_{n}}{c}$
for some $s\in\rti_{>0}$;
\item \label{enu:ConvConvex}$\setconv{a_{n}}{c}$ is $\rti$-convex, i.e.~if
also $y\in\setconv{a_{n}}{c}$, then $\forall t\in[0,1]:\ y+t(\bar{x}-y)\in\setconv{a_{n}}{c}$;
\item \label{enu:ConvConn}The set of convergence $\setconv{a_{n}}{c}$
is \emph{strongly connected}, i.e.~it is not possible to write it
as union of two non empty strongly disjoint sets, i.e.~such that
\begin{enumerate}[label=(\alph*)]
\item $A$, $B\subseteq\rti$, $A\ne\emptyset\ne B$,
\item $\exists\sup(A)$, $\exists\inf(B)$, $\sup(A)\le\inf(B)$,
\item $\setconv{a_{n}}{c}=A\cup B$,
\item $\exists m\in\N:\ B_{\diff\rho^{m}}(A)\cap B_{\diff\rho^{m}}(B)=\emptyset$.
\end{enumerate}
\end{enumerate}
\end{thm}

\begin{proof}
Without loss of generality we can assume $c=0$. From \cite[Lem.~5.(ii)]{MTAG},
we have either $\bar{x}=_{L}0$ or $|\bar{x}|>0$ for some $L\subzero I$.
The first case is actually impossible because $0\le|x|<|\bar{x}|=_{L}0$.
We can hence work only in the latter case $|\bar{x}|>0$. From Lem.~\ref{lem:KstrongBound},
we have $\forall^{0}\eps\,\forall n\in\N:\ |a_{n\eps}\bar{x}_{\eps}^{n}|\le K_{\eps}$.
Setting $h:=\left|\frac{x}{\bar{x}}\right|$, we have $h<1$ because
$|x|\in B_{|\bar{x}|}(0)$, and
\begin{equation}
\forall^{0}\eps\,\forall n\in\N:\ |a_{n\eps}x_{\eps}^{n}|=\left|a_{n\eps}\bar{x}_{\eps}^{n}\right|.\left|\frac{x_{\eps}}{\bar{x}_{\eps}}\right|^{n}<K_{\eps}h_{\eps}^{n}.\label{eq:km-bound}
\end{equation}

\noindent Thereby, $\sum_{n=N}^{M}|a_{n}x^{n}|\le\sum_{n=N}^{M}Kh^{n}$
for all $N$, $M\in\hypNs$. By the direct comparison test \cite[Thm. 22]{TiGi},
the HPS $\hypersum{\rho}{\sigma}a_{n}x^{n}$ converges absolutely
because $\hypersum{\rho}{\sigma}Kh^{n}$ converges since $\sigma\leq\rho^{*}$
and $h<1$. Finally, \cite[Thm.~74]{TI} yields that pointwise convergence
implies uniform convergence on functionally compact sets. This proves
\ref{enu:absConvUnif}.

\ref{enu:eventuallyBound}: From \eqref{eq:km-bound} it follows that
$\sum_{n=0}^{+\infty}|a_{n\eps}x_{\eps}^{n}|=:R_{\eps}$ converges
and is $\rho$-moderate. This implies condition \eqref{eq:eventBound}.

For \ref{enu:notBord}, it suffices to consider that $\hypersum{\rho}{\rho}\frac{(-1)^{n}}{n}$
converges (see \cite[Sec.~3.6]{TiGi}) whereas $\hypersum{\rho}{\rho}\frac{1}{n}$
does not by \cite[Thm.~18]{TiGi}. Note however, that for $x=1$,
we have $|x|=\radconv{\frac{1}{n}}$ so that condition Def.~\ref{def:setOfConv}.\ref{enu:ConvRadConv}
does not hold.

\ref{enu:underConv}: From the assumptions, $x\in B_{|\bar{x}-c|}(0)$,
$|\bar{x}-c|<\radconv{a_{n}}$, and hence Def.~\ref{def:setOfConv}.\ref{enu:ConvRadConv}
and Def.~\ref{def:setOfConv}.\ref{enu:ConvConv} follow. Note that
Def.~\ref{def:setOfConv}.\ref{enu:ConvFHPS} can be proved as above
from \eqref{eq:km-bound}. Finally, if $[\hat{x}_{\eps}]=x$ and $k\in\N_{>0}$,
we have
\begin{align}
\frac{\diff{}^{k}}{\diff{x}^{k}}\left(\sum_{n=0}^{+\infty}a_{n\eps}\hat{x}_{\eps}^{n}\right) & \le\sum_{n=k}^{+\infty}|a_{n\eps}|k!{n \choose k}\left|\frac{\hat{x}_{\eps}}{\bar{x}_{\eps}}\right|^{n-k}|\bar{x}_{\eps}|^{n-k}\nonumber \\
 & \le K_{\eps}|\bar{x}_{\eps}|^{-k}\sum_{n=k}^{+\infty}k!{n \choose k}\left|\frac{\hat{x}_{\eps}}{\bar{x}_{\eps}}\right|^{n-k}\in\R_{\rho},\label{eq:estimDer}
\end{align}
where we used Lem.~\ref{lem:KstrongBound}, and hence Def.~\ref{def:setOfConv}.\ref{enu:ConvDerMod}
also holds.

\ref{enu:underInterior}: For $s:=|\bar{x}|-|x|>0$ and $\hat{x}\in B_{s}(x)$,
we have $|\hat{x}|\le|\hat{x}-x|+|x|<s+|x|=|\bar{x}|$, and hence
$\hat{x}\in\setconv{a_{n}}{c}$ from \ref{enu:underConv}.

\ref{enu:ConvConvex}: Setting $\hat{x}:=y+t(\bar{x}-y)$, we have
$y\le\hat{x}\le\bar{x}$. We can use trichotomy law \cite[Lem.~7.(iii)]{MTAG}
to distinguish the cases $y=_{L}0$ or $y>_{L}0$ or $y>_{L}0$ for
$L\subzero I$. The latter has to be subdivided into the sub-cases
$\hat{x}>_{M}0$ or $\hat{x}=_{M}0$ or $\hat{x}<_{M}0$ with $M\subzero L$,
i.e.~using \cite[Lem.~7.(iii)]{MTAG} for the ring $\rti|_{L}$.
Finally, the latter of these sub-cases has to be further subdivided
into $\hat{x}>_{K}y$ or $\hat{x}<_{K}y$ or $\hat{x}=_{K}y$ with
$K\subzero M$. In all these cases we can prove Def.~\ref{def:setOfConv}
in the corresponding co-final set.

\ref{enu:ConvConn}: By contradiction, if $a\in A$ and $b\in B$,
then $x:=\frac{1}{2}\left(\sup(A)+\inf(B)\right)$ lies in the segment
$[a,b]\subseteq\setconv{a_{n}}{c}$ by \ref{enu:ConvConvex}. But
property $B_{\diff\rho^{m}}(A)\cap B_{\diff\rho^{m}}(B)$ implies
that $\sup(A)<\inf(B)$ and hence $x\notin A\cup B=\setconv{a_{n}}{c}$.
\end{proof}
In spite of Thm.~\ref{thm:absConvTopProp}.\ref{enu:underInterior},
it remains open the problem whether the set of convergence is always
a sharply open set or not. Using the previous theorem, this problem
depends, for each point $x$ in the set of convergence, on the existence
of a point $\bar{x}$ satisfying its assumptions. However, $x=1\in\setconv{\frac{\delta^{(n)}(0)}{n!}}{0}$
but Rem.~\ref{rem:eventBound}.\ref{enu:eventBoundDelta} shows that
$\left(\frac{\delta^{(n)}(0)}{n!}x^{n}\right)_{n\in\N}$ is not eventually
$\rti$-bounded in $\rcrhoc$, so that such a point $\bar{x}$ in
this case does not exist.
\begin{cor}
\label{cor:convTOsetconv}Let $\sigma\leq\rho^{*}$ and let $R$ be
the set of all the numbers of the form $s=|\bar{x}-c|$ for some $\bar{x}\in\rti$
satisfying:
\begin{enumerate}
\item \label{enu:supConv}${\textstyle \hypersum{\rho}{\sigma}a_{n}(\bar{x}-c)^{n}}$
is a convergent HPS,
\item \label{enu:supEventB}$(a_{n}(\bar{x}-c)^{n})_{n\in\N}$ is eventually
$\rti$-bounded in $\rcrhoc$.
\end{enumerate}
\noindent If $\exists\sup R=:r\in\rti$, then $B_{r}(c)\subseteq\setconv{a_{n}}{c}$,
the HPS $\hypersum{\rho}{\sigma}a_{n}(x-c)^{n}$ converges absolutely
for all $x\in B_{r}(c)$ and uniformly on every functionally compact
$K\fcmp\overline{B}_{r}(c)$.
\begin{proof}
Without loss of generality, we assume $c=0$, and let $x\in B_{r}(c)$.
Since $|x|<r$, by the definition of sharp supremum, (see \cite{MTAG})
there exist $s=\left|\bar{x}\right|$ such that $|x|<\left|\bar{x}\right|\leq r$
and such that \ref{enu:supConv} and \ref{enu:supEventB} hold. The
conclusions then follow by Thm.~\ref{thm:absConvTopProp}.
\end{proof}
\end{cor}

Property Thm.~\ref{thm:absConvTopProp}.\ref{enu:underConv} can
also be written as a characterization of the set of convergence:
\begin{cor}
\label{cor:setConv_sigma<=00003Drho*}Let $\sigma\leq\rho^{*}$, $(a_{n})_{{\rm c}}=[a_{n\eps}]_{{\rm c}}\in\rcrhoc$,
$c=[c_{\eps}]\in\rcrho$ such that $\hypersum{\rho}{\sigma}a_{n}(\bar{x}-c)^{n}$
is a convergent HPS whose sequence of summands $\left(a_{n}(\bar{x}-c)^{n}\right)_{n\in\N}$
is eventually $\rti$-bounded in $\rcrhoc$. If $x\in B_{|\bar{x}-c|}(0)$,
then $x=[x_{\eps}]\in\setconv{a_{n}}{c}$ if and only if
\[
\hypersum{\rho}{\sigma}a_{n}(x-c)^{n}=\left[\sum_{n=0}^{\infty}a_{n\eps}(x_{\eps}-c_{\eps})^{n}\right]\in\rcrho.
\]
\end{cor}

\subsection{Algebraic properties of hyper-power series}

In this section, we extend to HPS the classical results concerning
algebraic operations and composition of power series.
\begin{thm}
\label{thm:algProperties}Assume that $\hypersum{\rho}{\sigma}a_{n}\left(x-c\right){}^{n}$
and $\hypersum{\rho}{\sigma}b_{n}\left(x-c\right){}^{n}$ are two
convergent HPS, then:
\begin{enumerate}
\item \label{enu:prodScalar}For all $r\in\rti$, the product $r\cdot\hypersum{\rho}{\sigma}a_{n}\left(x-c\right){}^{n}$
is a convergent HPS with $\radconv{ra_{n}}\ge\radconv{a_{n}}$, and
\begin{equation}
r\cdot\hypersum{\rho}{\sigma}a_{n}\left(x-c\right){}^{n}=\hypersum{\rho}{\sigma}ra_{n}\left(x-c\right){}^{n}.\label{eq:prodScalar}
\end{equation}
\item \label{enu:sum_conv}The sum of these HPS is a convergent HPS with
\[
\radconv{a_{n}+b_{n}}\ge\min(\radconv{a_{n}},\radconv{b_{n}}),
\]
and
\begin{equation}
\hypersum{\rho}{\sigma}\left(a_{n}+b_{n}\right)\left(x-c\right)^{n}=\hypersum{\rho}{\sigma}a_{n}\left(x-c\right){}^{n}+\hypersum{\rho}{\sigma}b_{n}\left(x-c\right){}^{n}.\label{eq:sum}
\end{equation}
\item \label{enu:prod}For all $\bar{x}\in B_{|x-c|}(c)$, the product of
these HPS converges to their Cauchy product:
\begin{equation}
\left(\hypersum{\rho}{\sigma}a_{n}\left(\bar{x}-c\right){}^{n}\right)\cdot\left(\hypersum{\rho}{\sigma}b_{n}\left(\bar{x}-c\right){}^{n}\right)=\hypersum{\rho}{\sigma}\sum_{k=0}^{n}a_{k}b_{n-k}\left(\bar{x}-c\right){}^{n},\label{eq:CauchyProd}
\end{equation}
which is still a convergent HPS with radius of convergence greater
or equal to $\min(\radconv{a_{n}},\radconv{b_{n}})$.
\item \label{enu:division}Let $\left[a_{n\eps}\right]_{\text{c}}=(a_{n})_{{\rm c}}$
and $\left[b_{n\eps}\right]_{\text{c}}=(b_{n})_{{\rm c}}$ be representatives
of the coefficients of the given HPS. Assume that $b_{0}=[b_{0\eps}]\in\rti$
is invertible, and recursively define (for $\eps$ small) $d_{0\eps}:=\frac{a_{0\eps}}{b_{0\eps}}$,
\begin{equation}
d_{n\eps}:=\frac{1}{b_{0\eps}}\left(a_{n\eps}-\sum_{l=1}^{n}b_{l\eps}d_{n-l,\eps}\right)\qquad\forall n\in\N_{>0}.\label{eq:defDivCoeff}
\end{equation}
Then coefficients $(d_{n})_{{\rm c}}\in\rcrhoc$ define a convergent
HPS with radius of convergence greater or equal to $\min(\radconv{a_{n}},\radconv{b_{n}})$
such that for all $\bar{x}\in B_{|x-c|}(c)$, if $\hypersum{\rho}{\sigma}b_{n}\left(\bar{x}-c\right){}^{n}$
is invertible, then
\begin{equation}
\frac{\hypersum{\rho}{\sigma}a_{n}\left(\bar{x}-c\right){}^{n}}{\hypersum{\rho}{\sigma}b_{n}\left(\bar{x}-c\right){}^{n}}=\hypersum{\rho}{\sigma}d_{n}\left(\bar{x}-c\right){}^{n}.\label{eq:division}
\end{equation}
\end{enumerate}
\end{thm}

\begin{proof}
Equalities \eqref{eq:prodScalar} and \eqref{eq:sum} follow directly
from analogous properties of convergent hyperlimits, i.e.~\cite[Sec.~5.2]{MTAG}.
All the inequalities concerning the radius of convergence can be proved
in the same way from analogous results of the classical theory, because
of Def.~.\ref{enu:radConv}\ref{def:radCon}. For example, from Def.~\ref{def:setOfConv}.\ref{enu:ConvConv}
we have that both the ordinary series $\sum_{n=0}^{+\infty}a_{n\eps}(x_{\eps}-c_{\eps})^{n}$
and $\sum_{n=0}^{+\infty}b_{n\eps}(x_{\eps}-c_{\eps})^{n}$ converge.
Thereby, their sum $\sum_{n=0}^{+\infty}\left(a_{n\eps}+b_{n\eps}\right)(x_{\eps}-c_{\eps})^{n}$
converges with radius $\radconveps{a_{n}+b_{n}}\ge\min\left(\radconveps{a_{n}},\radconveps{b_{n}}\right)$.
To prove \eqref{eq:CauchyProd} (assuming that $\bar{x}$ lies in
the convergence set of the product HPS, see below), from Lem.~\ref{thm:absConvTopProp}
we have that both the series converge absolutely because $\bar{x}\in B_{|x-c|}(c)$.
We can hence apply the generalization of Mertens' theorem to hyperseries
(see \cite[Thm.~37]{TiGi}). To complete the proof of \ref{enu:prod},
we start by showing that the terms $\left(\sum_{k=0}^{n}a_{k\eps}b_{n-k,\eps}\right)_{n,\eps}$
defines coefficients for an HPS. Let $(a_{n})_{\text{{\rm c}}}=\left[a_{n\eps}\right]_{\text{{\rm c}}}$,
$(b_{n})_{\text{{\rm c}}}=\left[b_{n\eps}\right]_{\text{{\rm c}}}\in\rcrhoc$,
so that:

\begin{equation}
\exists Q_{1},R_{1}\in\N\,\forall^{0}\eps\,\forall n\in\N:\ \left|a_{n\eps}\right|\le\rho_{\eps}^{-nQ_{1}-R_{1}}.\label{eq:weakMod_a}
\end{equation}

\begin{equation}
\exists Q_{2},R_{2}\in\N\,\forall^{0}\eps\,\forall n\in\N:\ \left|b_{n\eps}\right|\le\rho_{\eps}^{-nQ_{2}-R_{2}}.\label{eq:weakMod_b}
\end{equation}

\noindent Without loss of generality we can assume $Q_{2}>Q_{1}$.
We have
\begin{align}
\left|\sum_{k=0}^{n}a_{k\eps}b_{n-k,\eps}\right| & \leq\sum_{k=0}^{n}|a_{k\eps}||b_{n-k,\eps}|\nonumber \\
 & \leq\sum_{k=0}^{n}\rho_{\eps}^{-kQ_{1}-R_{1}}.\rho_{\eps}^{-(n-k)Q_{2}-R_{2}}\nonumber \\
 & \leq\sum_{k=0}^{n}\rho_{\eps}^{-kQ_{1}+kQ_{2}-nQ_{2}-R_{1}-R_{2}}.\label{eq:a_kb_n-k_bound}
\end{align}

\noindent We have $\rho_{\eps}^{Q_{2}-Q_{1}}<1$ because $Q_{2}>Q_{1}$,
and hence

\[
\left|\sum_{k=0}^{n}a_{k\eps}b_{n-k,\eps}\right|\leq\frac{\rho_{\eps}^{-nQ_{2}-R_{1}-R_{2}}}{1-\rho_{\eps}^{-Q_{1}+Q_{2}}}\leq\rho_{\eps}^{-nQ-R},
\]
where $R:=R_{1}+R_{2}$ and for a suitable $Q\in\N$ (that can be
chosen uniformly with respect to $n\in\N$). Thereby, the product
HPS has well-defined coefficients and hence a suitable set of convergence.
Now, we want to show that $\bar{x}$ lies in this set of convergence.
Since Def.~\ref{def:setOfConv}.\ref{enu:ConvRadConv} clearly holds
and Def.~\ref{def:setOfConv}.\ref{enu:ConvConv} follows from Mertens'
Theorem (both \cite[Thm.~37]{TiGi} and the classical version), it
remains to prove that we actually have a formal HPS (Def.~\ref{def:setOfConv}.\ref{enu:ConvFHPS})
and moderateness of derivatives (Def.~\ref{def:setOfConv}.\ref{enu:ConvDerMod}).
The latter follows by the general Leibniz rule for the $k$-th derivative
of a product. For the former one, without loss of generality we can
assume $c=0$; let $\left(M_{\eps}\right)$, $\left(N_{\eps}\right)\in\N_{\sigma}$,
then for suitable $\left(\bar{M}_{\eps}\right)$, $\left(\hat{M}_{\eps}\right)\in\N_{\sigma}$
and $\left(\bar{N}_{\eps}\right)$, $\left(\hat{N}_{\eps}\right)\in\N_{\sigma}$
such that $M_{\eps}=\bar{M}_{\eps}+\hat{M}_{\eps}$ and $N_{\eps}=\bar{N}_{\eps}+\hat{N}_{\eps}$,
we have
\begin{equation}
\left(\sum_{n=N_{\eps}}^{M_{\eps}}\sum_{k=0}^{n}a_{n\eps}b_{n-k,\eps}\hat{x}_{\eps}^{n}\right)=\left(\sum_{n=\bar{N}_{\eps}}^{\bar{M}_{\eps}}a_{n\eps}\hat{x}_{\eps}^{n}\right)\cdot\left(\sum_{n=\hat{N}_{\eps}}^{\hat{M}_{\eps}}b_{n,\eps}\hat{x}_{\eps}^{n}\right),\label{eq:prodMod}
\end{equation}
and thereby Def.~\ref{def:setOfConv}.\ref{enu:ConvFHPS} follows.

\noindent \ref{enu:division}: To prove that $(d_{n})_{{\rm c}}\in\rcrhoc$,
without loss of generality, we can assume in \eqref{eq:weakMod_a}
and \eqref{eq:weakMod_b} that $Q_{1}=Q_{2}=:\hat{Q}>R_{1}=R_{2}=:\hat{R}$
and $\hat{Q}>0$. By induction on $n\in\N$, we what to prove that
\begin{equation}
\forall^{0}\eps\,\forall n\in\N:\ \left|d_{n\eps}\right|\le\rho_{\eps}^{-n\hat{Q}-\hat{Q}}.\label{eq:divInd}
\end{equation}
For $n=0$, we have $\left|d_{0\eps}\right|=\left|\frac{a_{0\eps}}{b_{0\eps}}\right|\le\rho_{\eps}^{-\hat{R}+\hat{R}}\le\rho_{\eps}^{-\hat{Q}}$
for all $\eps$ because $\hat{Q}>0$. For the inductive step, we assume
\eqref{eq:divInd} and use the recursive definition \eqref{eq:defDivCoeff}:

\begin{align*}
\left|d_{n+1,\eps}\right| & \le\left|\frac{a_{n+1,\eps}}{b_{0\eps}}\right|+\left|\frac{\sum_{l=1}^{n+1}b_{l\eps}d_{n-l,\eps}}{b_{0\eps}}\right|\\
 & \le\rho_{\eps}^{-(n+1)\hat{Q}-\hat{R}}\cdot\rho_{\eps}^{\hat{R}}+\sum_{l=1}^{n+1}\rho_{\eps}^{-l\hat{Q}-\hat{R}}\cdot\rho_{\eps}^{-(n-l)\hat{Q}-\hat{Q}}\cdot\rho_{\eps}^{\hat{R}}\\
 & =\rho_{\eps}^{-n\hat{Q}-\hat{Q}}+\rho_{\eps}^{-n\hat{Q}-\hat{Q}}\le2\rho_{\eps}^{-n\hat{Q}-\hat{Q}}.
\end{align*}
We have $2\rho_{\eps}^{-n\hat{Q}-\hat{Q}}\le\rho_{\eps}^{-(n+1)\hat{Q}-\hat{Q}}$
if and only if $2\le\rho_{\eps}^{-\hat{Q}}$, which holds for $\eps$
small (independently by $n$). Finally, equality \eqref{eq:division}
can be proved as we did above for the product because $\bar{x}\in B_{|x-c|}(c)$
and

\begin{equation}
\hypersum{\rho}{\sigma}a_{n}\left(\bar{x}-c\right){}^{n}=\hypersum{\rho}{\sigma}d_{n}\left(\bar{x}-c\right){}^{n}.\hypersum{\rho}{\sigma}b_{n}\left(\bar{x}-c\right){}^{n}.\label{eq:dev_mult}
\end{equation}
From this equality, it also follows Def.~\ref{def:convHps}.\ref{enu:ConvFHPS}
because the product of a non-moderate net (on a co-final set) by a
moderate net cannot yield a moderate net. Finally, as above, moderateness
of derivatives follows from Mertens' theorem and the $k$-th derivative
of the quotient.
\end{proof}
The following theorem concerns the composition of HPS:
\begin{thm}
\label{thm:composition} Let $(a_{n})_{\text{{\rm c}}}=\left[a_{n\eps}\right]_{\text{{\rm c}}}$,
$(b_{n})_{\text{{\rm c}}}=\left[b_{n\eps}\right]_{\text{{\rm c}}}\in\rcrhoc$
be coefficients for HPS. Set $f(y):=\hypersum{\rho}{\sigma}a_{n}(y-b_{0}){}^{n}$
for all $y\in\setconv{a_{n}}{b_{0}}$ and $g(x):=\hypersum{\rho}{\sigma}b_{n}(x-c){}^{n}$
for all $x\in\setconv{b_{n}}{c}$. Set
\begin{align*}
c_{0\eps} & :=a_{0\eps}\\
c_{n\eps} & :=\sum_{k=0}^{+\infty}a_{k\eps}\sum_{m_{1}+\ldots+m_{k}=n}b_{m_{1}\eps}\cdot\ldots\cdot b_{m_{k}\eps}\qquad\forall n\in\N_{>0}.
\end{align*}
If $x\in\setconv{b_{n}}{c}$ and $g(x)\in\setconv{a_{n}}{b_{0}}$,
then $f(g(x))=\hypersum{\rho}{\sigma}c_{n}(x-c)^{n}$ is a convergent
HPS.
\end{thm}

\begin{proof}
Since $\left[a_{n\eps}\right]_{c}$, $\left[b_{n\eps}\right]_{\text{{\rm c}}}\in\rcrhoc$
, we can assume that both \eqref{eq:weakMod} and \eqref{eq:strongEq}
hold with $\hat{Q}=Q_{1}=Q_{2}>0$ and $\hat{R}=R_{1}=R_{2}>0$. We
have
\begin{align*}
\left|\sum_{k=0}^{n}a_{k\eps}\sum_{m_{1}+\ldots+m_{k}=n}b_{m_{1}\eps}\cdot\right. & \left.\ldots\cdot b_{m_{k}\eps}\right|\leq\sum_{k=0}^{n}|a_{k\eps}|\sum_{m_{1}+\ldots+m_{k}=n}|b_{m_{1}\eps}|\cdot\ldots\cdot|b_{m_{k}\eps}|\\
 & \leq\sum_{k=0}^{n}\rho_{\eps}^{-k\hat{Q}-\hat{R}}\sum_{m_{1}+\ldots+m_{k}=n}\rho_{\eps}^{-m_{1}\hat{Q}-\hat{R}}\cdot\ldots\cdot\rho_{\eps}^{-m_{k}\hat{Q}-\hat{R}}\\
 & =\sum_{k=0}^{n}\rho_{\eps}^{-k\hat{Q}-\hat{R}}\sum_{m_{1}+\ldots+m_{k}=n}\rho_{\eps}^{-n\hat{Q}-k\hat{R}}\\
 & =\rho_{\eps}^{-\hat{R}}+\sum_{k=1}^{n}\rho_{\eps}^{-k\hat{Q}-\hat{R}}\sum_{m_{1}+\ldots+m_{k}=n}\rho_{\eps}^{-n\hat{Q}-k\hat{R}}\\
 & =\rho_{\eps}^{-\hat{R}}+\sum_{k=1}^{n}\rho_{\eps}^{-k\hat{Q}-\hat{R}-n\hat{Q}-k\hat{R}}{n+k-1 \choose k-1}\\
 & \le\rho_{\eps}^{-\hat{R}}+2^{2n}\rho_{\eps}^{-\hat{R}-n\hat{Q}}\cdot\frac{1-\rho_{\eps}^{-(n+1)(\hat{Q}+\hat{R})}}{1-\rho_{\eps}^{-\hat{Q}-\hat{R}}}=:[*].
\end{align*}
For $\eps$ small, we have $\frac{4}{\rho_{\eps^{-1}}}\le1$, hence
$\frac{2^{2n}}{\rho_{\eps}^{-n}}\le1$ for the same $\eps$ and for
all $n\in\N$. Now, take $\eps$ small so that also $\frac{1}{1-\rho_{\eps}^{-\hat{Q}-\hat{R}}}\le1$,
and $\frac{1}{\rho_{\eps}^{-1}}\le\frac{1}{3}$. We hence have
\[
[*]\le\rho_{\eps}^{-\hat{R}}+\rho_{\eps}^{-n\hat{Q}-\hat{R}-n}+\rho_{\eps}^{-2n\hat{Q}-n\hat{R}-2\hat{R}-n}.
\]
Since
\begin{align*}
\frac{\rho_{\eps}^{-\hat{R}}}{\rho_{\eps}^{-n(2\hat{Q}+\hat{R}+1)-2\hat{R}-1}} & \le\frac{1}{\rho_{\eps}^{-1}}\le\frac{1}{3}\\
\frac{\rho_{\eps}^{-n\hat{Q}-\hat{R}-n}}{\rho_{\eps}^{-n(2\hat{Q}+\hat{R}+1)-2\hat{R}-1}} & \le\frac{1}{\rho_{\eps}^{-1}}\le\frac{1}{3}\\
\frac{\rho_{\eps}^{-2n\hat{Q}-n\hat{R}-2\hat{R}-n}}{\rho_{\eps}^{-n(2\hat{Q}+\hat{R}+1)-2\hat{R}-1}} & \le\frac{1}{\rho_{\eps}^{-1}}\le\frac{1}{3},
\end{align*}
we finally get
\[
\forall^{0}\eps\,\forall n\in\N:\ \left|\sum_{k=0}^{n}a_{k\eps}\sum_{m_{1}+\ldots+m_{k}=n}b_{m_{1}\eps}\cdot\ldots\cdot b_{m_{k}\eps}\right|\le\rho_{\eps}^{-n(2\hat{Q}+\hat{R}+1)-2\hat{R}-1},
\]
which proves that $(c_{n\eps})_{n,\eps}$ defines coefficients for
an HPS. To prove that $x\in\setconv{c_{n}}{c}$, we can proceed as
follows: Def.~\ref{def:setOfConv}.\ref{enu:ConvRadConv} can be
proved like in the classical case; Def.~\ref{def:setOfConv}.\ref{enu:ConvFHPS}
is a consequence of composition of polynomials if $M_{\eps}<+\infty$
or it can be proved proceeding like in the case of composition of
GSF if $M_{\eps}=+\infty$: Def.~\ref{def:setOfConv}.\ref{enu:ConvConv}
and Def.~\ref{def:setOfConv}.\ref{enu:ConvDerMod} can be proved
like for GSF (see \cite{TI} and Thm.~\ref{thm:GRAFareGSF} below).
\end{proof}

\section{Generalized real analytic functions and their calculus}

A direct consequence of Def.~\ref{def:setOfConv} of set of convergence
is the following
\begin{thm}
\label{thm:GRAFareGSF} Let $\left[a_{n\eps}\right]_{\text{{\rm c}}}=\left(a_{n}\right)_{\text{{\rm c}}}\in\rcrhoc$
and $c=[c_{\eps}]\in\rcrho$. Set $f(x):=\hypersum{\rho}{\sigma}a_{n}(x-c)^{n}=\left[\sum_{n=0}^{\infty}a_{n\eps}(x_{\eps}-c_{\eps})^{n}\right]=:\left[v_{\eps}(x_{\eps})\right]$
for all $x=[x_{\eps}]\in\setconv{a_{n}}{c}$. Then $f\in\gsf\left(\setconv{a_{n}}{c},\rcrho\right)$
is a GSF defined by $\left(v_{\eps}\right)$.
\end{thm}

Before defining the notion of GRAF, we need to prove that the derived
HPS has the same set of convergence of the original HPS:
\begin{thm}
\label{thm:dertoconv}Assume $\sigma\leq\rho^{*}$, $\left(a_{n}\right)_{\text{{\rm c}}}\in\rcrhoc$
and $c\in\rcrho$. Then the set of convergence of the derived series
$\parthypersumarg{\rho}{\sigma}{n}{>0}na_{n}(x-c)^{n-1}=\hypersum{\rho}{\sigma}(n+1)a_{n+1}(x-c)^{n}$
is the same as the set of convergence of the original HPS $\hypersum{\rho}{\rho}a_{n}(x-c)^{n}$
. Thereby, recursively, all the derivatives has the same set of convergence
of the original HPS and define a GSF.
\end{thm}

\begin{proof}
By Def.~\ref{def:radCon}.\ref{enu:radConv} of radius of convergence
and the classical theory, we have
\begin{align*}
\radconveps{a_{n}} & =\left(\limsup_{n\to+\infty}\left|a_{n\eps}\right|^{1/n}\right)^{-1}=\left(\limsup_{n\to+\infty}\left|(n+1)a_{n+1,\eps}\right|^{1/n+1}\right)^{-1}\\
 & =\radconveps{(n+1)a_{n+1}},
\end{align*}
so Def.~\ref{def:setOfConv}.\ref{enu:ConvRadConv} for the original
HPS and the derived one are equivalent. From the condition $\left[a_{n\eps}\cdot\left(x_{\eps}-c_{\eps}\right)^{n}\right]_{\text{{\rm s}}}\in\hps{x-c}$
and $\sigma\le\rho^{*}$, in the usual way it follows that $\left[(n+1)a_{n+1,\eps}\cdot\left(x_{\eps}-c_{\eps}\right)^{n}\right]_{\text{{\rm s}}}\in\hps{x-c}$.
Vice versa, from $(n+1)\left|a_{n+1,\eps}\right|\ge\left|a_{n+1,\eps}\right|$
the opposite implication follows. The condition Def.~\ref{def:setOfConv}.\ref{enu:ConvDerMod}
about moderateness of derivatives for the original HPS clearly implies
the analogue condition for the derived one. For the opposite inclusion,
we can distinguish the case $x=_{\text{s}}c$ or $|x-c|>0$, the former
one being trivial. We have
\begin{align*}
\left|\sum_{n=1}^{+\infty}a_{n\eps}n(x_{\eps}-c_{\eps})^{n-1}\right| & =|x_{\eps}-c_{\eps}|^{-1}\left|\sum_{n=1}^{+\infty}a_{n\eps}n(x_{\eps}-c_{\eps})^{n}\right|\\
 & \ge|x_{\eps}-c_{\eps}|^{-1}\left|\sum_{n=0}^{+\infty}a_{n\eps}(x_{\eps}-c_{\eps})^{n}\right|,
\end{align*}
so that also the net $\left(\sum_{n=0}^{+\infty}a_{n\eps}(x_{\eps}-c_{\eps})^{n}\right)\in\R_{\rho}$
if the derivative is moderate.
\end{proof}
Thm.~\ref{thm:GRAFareGSF} motivates the following definition:
\begin{defn}
\label{def:GRAF}Let $\sigma\le\rho^{*}$ and $U$ be a sharply open
set of $\rcrho$, then we say that $f$ \emph{is a GRAF on} $U$ (with
respect to $\rho$, $\sigma$), and we write $f\in\gsft(U,\rcrho)$
if $f:U\ra\rti$ and for all $c\in U$ we can find $s\in\rcrho_{>0}$,
$(a_{n})_{c}\in\rcrhoc$ such that
\begin{enumerate}
\item \label{enu:intGRAF}$(c-s,c+s)\subseteq U\cap\setconv{a_{n}}{c}$,
\item \label{enu:GRAF=00003Df}$f(x)=\hypersum{\rho}{\sigma}a_{n}\left(x-c\right){}^{n}$
for all $x\in(c-s,c+s)$.
\end{enumerate}
\noindent Moreover, we say that $f:\rcrho\rightarrow\rcrho$ \emph{is
an entire function} (with respect to $\rho$, $\sigma$) if we can
find $c\in\rcrho$ and $(a_{n})_{c}\in\rcrhoc$ such that
\begin{enumerate}[resume]
\item \label{enu:entireR}$\rcrho=\setconv{a_{n}}{c}$,
\item \label{enu:entiref=00003D}$f(x)=\hypersum{\rho}{\sigma}a_{n}(x-c)^{n}$
for all $x\in\rcrho$.
\end{enumerate}
\noindent We also say that $f$ \emph{is entire at} $c$ if \ref{enu:entireR}
and \ref{enu:entiref=00003D} hold.

\end{defn}

\begin{example}
\noindent \ 
\begin{enumerate}[label=(\alph*)]
\item \label{enu:HPSdefinesGRAF}Clearly, if $(a_{n})_{c}\in\rcrhoc$,
$c\in\rti$, and we set $f(x)=\hypersum{\rho}{\sigma}a_{n}\left(x-c\right){}^{n}$,
then $f$ is a GRAF on the interior points of the set of convergence
$\setconv{a_{n}}{c}$. Vice versa, if $f\in\gsft(U,\rti)$, then $U$
is contained in the union of all the sharp interior sets $\text{int}\left(\setconv{a_{n}}{c}\right)$,
because of condition \ref{enu:intGRAF}.
\item \label{enu:deltaGRAF}Example \ref{exa:Dirac-delta} shows that Dirac
$\delta$ is entire at $0$ but it is not at any $c\in\rcrho$ such
that $\left|c\right|\geq s\in\rcrho$ for some $s$.
\item \label{enu:flatGRAF}Example \ref{exa:flatPoint} of a function $f$
with a flat point shows that $f$ is a GRAF, but if $c=0$, then $s\in\rti_{>0}$
satisfying condition \ref{enu:intGRAF} is infinitesimal, whereas
if $c\gg0$, then $s\gg0$ is finite, and these two types of set of
convergence are always disjoint.
\end{enumerate}
\end{example}

\begin{cor}
\noindent \label{cor:derGRAF}Let $\sigma\le\rho^{*}$, $U\subseteq\rti$
be a sharply open set and $f\in\gsft(U,\rcrho)$, then also $f'\in\gsft(U,\rcrho)$
and it can be computed with the derived HPS.
\end{cor}

\noindent Because of our definition Def.~\ref{def:setOfConv} of
set of convergence, several classical results can be simply translated
in our setting considering the real analytic function that defines
a given GRAF.
\begin{thm}
\label{thm:uniqueness}Let $\sigma\le\rho^{*}$, $\left(a_{n}\right)_{\text{\emph{c}}}\in\rcrhoc$,
$c\in\rti$, and set $f(x)=\hypersum{\rho}{\sigma}a_{n}(x-c)^{n}$
for all interior points $x\in\setconv{a_{n}}{c}$, then $a_{k}=\frac{f^{(k)}(c)}{k!}$
for all $k\in\N$.
\end{thm}

\begin{proof}
From Cor.~\ref{cor:derGRAF}, we have $f^{(k)}(x)=\left[\sum_{n=k}^{\infty}a_{n\eps}k!{n \choose k}(x_{\eps}-c_{\eps})^{n-k}\right]$
for all the interior points $x\in\setconv{a_{n}}{c}$. For $x=c$
(which is always a sharply interior point because of Thm.~\ref{thm:setConvNotTrivial})
this yields the conclusion.
\end{proof}
\begin{cor}
\label{cor:TaylorCoeff}Let $\sigma\le\rho^{*}$, $U$ be a sharply
open set of $\rcrho$, and $f\in\gsft(U,\rcrho)$. Then for all $c\in U$
the Taylor coefficients $\left(\frac{f^{(n)}(c)}{n!}\right)_{\text{c}}\in\rcrhoc$.
\end{cor}

\noindent The definition of $1$-dimensional integral of GSF by using
primitives, allows us to get a simple proof of the term by term integration
of GRAF:
\begin{thm}
\label{thm:integration}In the assumptions of the previous theorem,
set
\[
F(x):=\hypersum{\rho}{\sigma}\frac{a_{n}(x-c)^{n+1}}{n+1}
\]
for all the interior points $x\in\setconv{a_{n}}{c}$. Then $F(x)=\int_{c}^{x}f(x)\,\diff{x}$
and $F$ is a GRAF on the interior points of $\setconv{a_{n}}{c}$.
\end{thm}

\begin{proof}
The proof that $\hypersum{\rho}{\sigma}\frac{a_{n}(x-c)^{n+1}}{n+1}$
is a convergent HPS with the same set of convergence of $f$ can be
done as in Thm.~\ref{thm:dertoconv}, and hence $F$ is a GRAF on
the interior points of $\setconv{a_{n}}{c}$. The remaining part of
the proof is straightforward by using Cor.~\ref{cor:derGRAF}, so
that $F'(x)=f\left(x\right)$ and $F(c)=0$. and using \cite[Thm.~42, Def.~43]{TI}.
\end{proof}
\noindent We close this section by first noting that, differently
with respect to the classical theory, if $f(x)=\hypersum{\rho}{\sigma}a_{n}(x-c)^{n}$
for all $x\in\setconv{a_{n}}{c}$, and we take another point $\bar{c}\in\setconv{a_{n}}{c}$,
we do not have that $(\bar{c}-\radconv{a_{n}}+|c-\bar{c}|,\bar{c}+\radconv{a_{n}}-|c-\bar{c}|)\subseteq\setconv{a_{n}}{c}$;
in fact for $c=\bar{c}$ this would yield the false equality $(c-\radconv{a_{n}},c+\radconv{a_{n}})=\setconv{a_{n}}{c}$.
On the other hand, in the following result we show that $\setconv{\frac{f^{(n)}(\bar{c})}{n!}}{\bar{c}}\subseteq\setconv{\frac{f^{(n)}(c)}{n!}}{c}$:
\begin{thm}
\label{thm:twoPointsInSetConv}In the assumptions of Thm.~\ref{thm:uniqueness},
if $\bar{c}\in\setconv{\frac{f^{(n)}(c)}{n!}}{c}$, then $\setconv{\frac{f^{(n)}(\bar{c})}{n!}}{\bar{c}}\subseteq\setconv{\frac{f^{(n)}(c)}{n!}}{c}$.
\end{thm}

\begin{proof}
In fact, since $\bar{c}\in\setconv{\frac{f^{(n)}(c)}{n!}}{c}$, we
have
\[
f^{(n)}(\bar{c})=\ \parthypersumarg{\rho\,\,}{\sigma}{m}{\ge n}\frac{f^{(m)}(c)}{m!}n!{m \choose n}(x-c)^{m-n}.
\]
Thereby, if $x\in\setconv{\frac{f^{(n)}(\bar{c})}{n!}}{\bar{c}}$,
we have
\begin{align*}
f(x) & =\hypersum{\rho}{\sigma}\frac{f^{(n)}(\bar{c})}{n!}(x-\bar{c})^{n}\\
 & =\hypersum{\rho}{\sigma}\frac{(x-\bar{c})^{n}}{n!}\cdot\ \parthypersumarg{\rho\,\,\,}{\sigma}{m}{\ge n}\frac{f^{(m)}(c)}{m!}n!{m \choose n}(x-c)^{m-n}\\
 & =\left[\sum_{n=0}^{+\infty}\frac{(x_{\eps}-\bar{c}_{\eps})^{n}}{n!}\sum_{m\ge n}^{+\infty}\frac{f_{\eps}^{(m)}(c_{\eps})}{m!}n!{m \choose n}(x_{\eps}-c_{\eps})^{m-n}\right].
\end{align*}
Therefore, the usual proof, see e.g.~\cite{KrPa02}, yields
\[
\sum_{n=0}^{+\infty}\frac{(x_{\eps}-\bar{c}_{\eps})^{n}}{n!}\sum_{m\ge n}^{+\infty}\frac{f_{\eps}^{(m)}(c_{\eps})}{m!}n!{m \choose n}(x_{\eps}-c_{\eps})^{m-n}=\sum_{n=0}^{+\infty}\frac{f_{\eps}^{(n)}(c)}{n!}(x_{\eps}-c_{\eps})^{n}
\]
and hence $f(x)=\hypersum{\rho}{\sigma}\frac{f^{(n)}(c)}{n!}(x-c)^{n}=\left[\sum_{n=0}^{+\infty}\frac{f_{\eps}^{(n)}(c)}{n!}(x_{\eps}-c_{\eps})^{n}\right]$,
which implies the conclusion.
\end{proof}

\section{Characterization of generalized real analytic functions, inversion
and identity principle}

The classical characterization of real analytic functions by the growth
rate of the derivatives establishes a difference between GRAF and
Colombeau real analytic functions:
\begin{thm}
\label{thm:growthDers}Let $\sigma\le\rho^{*}$, $U$ be a sharply
open set of $\rcrho$, and $f\in\gsf(U,\rti)$ be a GSF defined by
the net $(f_{\eps})$. Then $f\in\gsft(U,\rti)$ if and only if for
each $c\in U$ there exist $s=[s_{\eps}]$, $C=\left[C_{\eps}\right]$,
$R=\left[R_{\eps}\right]\in\rti_{>0}$ such that $B_{s}(c)\subseteq U$
and
\begin{equation}
\forall[x_{\eps}]\in B_{s}(c)\,\forall^{0}\eps\,\forall n\in\N:\ \left|f_{\eps}^{(n)}(x_{\eps})\right|\le C_{\eps}\frac{n!}{R_{\eps}^{n}}.\label{eq:growthDers}
\end{equation}
\end{thm}

\begin{proof}
We prove that condition \eqref{eq:growthDers} is necessary. For $c\in U$,
we have $f(x)=\hypersum{\rho}{\sigma}\frac{f^{(n)}(c)}{n!}\left(x-c\right){}^{n}$
for all $x\in(c-\bar{s},c+\bar{s})$ for some $\bar{s}>0$ from Def.~\ref{def:GRAF}
and Thm.~\ref{thm:uniqueness}. We first note that condition \eqref{eq:growthDers}
can also be formulated as an inequality in $\rcrhoc$ and as such
it does not depend on the representatives involved. Therefore, from
Thm.~\ref{thm:GRAFareGSF} and Thm.~\ref{thm:uniqueness}, without
loss of generality, we can assume that the given net $(f_{\eps})$
is of real analytic functions satisfying $f_{\eps}(x)=\sum_{n=0}^{+\infty}\frac{f_{\eps}^{(n)}(c_{\eps})}{n!}(x-c_{\eps})^{n}$
for all $x\in(c_{\eps}-\radconveps{a_{n}},c_{\eps}+\radconveps{a_{n}})$.
From Lem.~\ref{lem:eventBoundNear_c}, locally the Taylor summands
$\left(\frac{f^{(n)}(c)}{n!}(\bar{x}-c)^{n}\right)_{n\in\N}$ are
eventually $\rti$-bounded in $\rcrhoc$ if $\bar{x}$ is sufficiently
near to $c=[c_{\eps}]$, i.e.~there exists $\sigma\in\rti_{>0}$
such that for each $\bar{x}=[\bar{x}_{\eps}]\in B_{\sigma}(c)$ we
have
\begin{equation}
\forall^{0}\eps\,\forall j\in\N:\ \left|\frac{f_{\eps}^{(j)}(c_{\eps})}{j!}(\bar{x}_{\eps}-c_{\eps})^{j}\right|\le K_{\eps},\label{eq:charEvBo}
\end{equation}
for some $K=[K_{\eps}]\in\rti$. Set $s:=\frac{1}{2}\min(\sigma,\bar{s})\in\R_{>0}$
and $S:=|\bar{x}-c|$, where $\bar{x}$ is any point such that $s<|\bar{x}-c|<\sigma$,
so that $0<\frac{s}{S}<1$ and from \eqref{eq:charEvBo} we obtain
\begin{equation}
\forall^{0}\eps\,\forall j\in\N:\ \left|f_{\eps}^{(j)}(c_{\eps})\right|\le K_{\eps}\frac{j!}{S_{\eps}^{j}}.\label{eq:derMaj}
\end{equation}
For each $[x_{\eps}]\in B_{s}(c)$, we have
\begin{align*}
f_{\eps}^{(n)}(x_{\eps}) & =\sum_{j=n}^{+\infty}\frac{f_{\eps}^{(j)}(c_{\eps})}{j!}n!{j \choose n}(x_{\eps}-c_{\eps})^{j-n},
\end{align*}
and hence from \eqref{eq:derMaj}:
\begin{align*}
\left|\frac{f_{\eps}^{(n)}(x_{\eps})}{n!}\right| & \le\sum_{j=n}^{+\infty}K_{\eps}{j \choose n}\frac{|x_{\eps}-c_{\eps}|^{j-n}}{S_{\eps}^{j}}\\
 & \le\frac{K_{\eps}}{S_{\eps}^{n}}\sum_{j=n}^{\infty}{j \choose n}\left(\frac{s_{\eps}}{S_{\eps}}\right)^{j-n}\\
 & =\frac{K_{\eps}}{S_{\eps}^{n}}\cdot\frac{1}{\left(1-\frac{s_{\eps}}{S_{\eps}}\right)^{n+1}}=\frac{K_{\eps}}{\left(1-\frac{s_{\eps}}{S_{\eps}}\right)}\cdot\frac{1}{\left(S_{\eps}\left(1-\frac{s_{\eps}}{S_{\eps}}\right)\right)^{n}},
\end{align*}
which is our claim for $C:=\frac{K}{1-\frac{s}{S}}$ and $R:=S\left(1-\frac{s}{S}\right)$.
Note that, differently with respect to the case of Colombeau real
analytic functions \cite{PiScVa09}, not necessarily the constant
$\frac{1}{R}$ is finite, e.g.~if $s\approx S$.

We now prove that the condition is sufficient. Let $c=[c_{\eps}]\in U$
and $s=[s_{\eps}]$, $C=[C_{\eps}]$, $R=[R_{\eps}]\in\rti_{>0}$
be the constants satisfying \eqref{eq:growthDers}. Set $\bar{s}:=\frac{1}{2}\min(s,R,\radconv{\frac{f^{(n)}(c)}{n!}})$
and take $x\in B_{\bar{s}}(c)$. We first prove the equality $f(x)=\hypersum{\rho}{\sigma}\frac{f^{(n)}(c)}{n!}(x-c)^{n}$.
Let $N=[N_{\eps}]\in\hypNs$, with $N_{\eps}\in\N$. For all $\eps$,
from Taylor's formula for the smooth $f_{\eps}$, we have
\[
\left|\sum_{n=0}^{N}\frac{f^{(n)}(c)}{n!}(x-c)^{n}-f(x)\right|=\left[\left|\frac{f_{\eps}^{(N_{\eps}+1)}(\xi_{\eps})}{(N_{\eps}+1)!}(x_{\eps}-c_{\eps})^{N_{\eps}+1}\right|\right]
\]
for some $t_{\eps}\in[0,1]_{\R}$ and for $\xi_{\eps}:=(1-t_{\eps})c_{\eps}+t_{\eps}x_{\eps}$.
Since $\left|\xi_{\eps}-c_{\eps}\right|=t_{\eps}\left|x_{\eps}-c_{\eps}\right|<\bar{s}_{\eps}<s_{\eps}$,
we can apply \eqref{eq:growthDers} and get $\forall^{0}\eps\,\forall n\in\N:\ \left|\frac{f_{\eps}^{(n)}(\xi_{\eps})}{n!}\right|\le\frac{C_{\eps}}{R_{\eps}^{n}}$.
Thereby, for these small $\eps$ and for $n=N_{\eps}+1$ we obtain
\[
\left|\sum_{n=0}^{N}\frac{f^{(n)}(c)}{n!}(x-c)^{n}-f(x)\right|\le C\left(\frac{\bar{s}}{R}\right)^{N+1},
\]
and hence the claim follows by $\hyperlim{\rho}{\sigma}\left(\frac{\bar{s}}{R}\right)^{N+1}=0$.\\
Now, we prove that $\hypersum{\rho}{\sigma}\frac{f^{(n)}(c)}{n!}(x-c)^{n}=\left[\sum_{n=0}^{+\infty}\frac{f_{\eps}^{(n)}(c_{\eps})}{n!}(x_{\eps}-c_{\eps})^{n}\right]$.
In fact, once again from \eqref{eq:growthDers} we have
\begin{align*}
\left|\sum_{n=0}^{N}\frac{f^{(n)}(c)}{n!}(x-c)^{n}-\left[\sum_{n=0}^{+\infty}\frac{f_{\eps}^{(n)}(c_{\eps})}{n!}(x_{\eps}-c_{\eps})^{n}\right]\right| & =\left[\left|\sum_{n=N_{\eps}+1}^{+\infty}\frac{f_{\eps}^{(n)}(c_{\eps})}{n!}(x_{\eps}-c_{\eps})^{n}\right|\right]\\
 & \le\left[\sum_{n=N_{\eps}+1}^{+\infty}\frac{C_{\eps}}{R_{\eps}^{n}}\left|x_{\eps}-c_{\eps}\right|^{n}\right]\\
 & \le C\cdot\ \ \parthypersumarg{\rho}{\sigma}{n}{\ge N+1}\left(\frac{\bar{s}}{R}\right)^{n}\to0
\end{align*}
because $\bar{s}<R$ and hence $\hypersum{\rho}{\sigma}\left(\frac{\bar{s}}{R}\right)^{n}$
converges. Finally, take $\bar{x}\in B_{\bar{s}}(c)$ such that $|x-c|<|\bar{x}-c|$.
As above, we can prove that $\hypersum{\rho}{\sigma}\frac{f^{(n)}(c)}{n!}(\bar{x}-c)^{n}$
converges; moreover from \eqref{eq:growthDers} we also have $\forall^{0}\eps\,\forall n\in\N:\ \left|\frac{f_{\eps}^{(n)}(c_{\eps})}{n!}(\bar{x}_{\eps}-c_{\eps})^{n}\right|\le C_{\eps}\left(\frac{\bar{s}_{\eps}}{R_{\eps}}\right)^{n}\le\frac{C_{\eps}}{1-\frac{\bar{s}_{\eps}}{R_{\eps}}}$.
This proves that $\left(\frac{f^{(n)}(c)}{n!}(\bar{x}-c)^{n}\right)_{n\in\N}$
is eventually $\rti$-bounded in $\rcrhoc$ and hence $x\in\setconv{\frac{f^{(n)}(c)}{n!}}{c}$
by Cor.~\ref{cor:setConv_sigma<=00003Drho*}.
\end{proof}
As we have already noted in this proof, differently with respect to
the definition of Colombeau real analytic function \cite{PiScVa09},
we have that, generally speaking, $\frac{1}{R}\in\rti$ is not finite.
For example, for $f=\delta$ at $c=0$, we have $\left|\frac{\delta_{\eps}^{(n)}(x_{\eps})}{n!}\right|=\left|\frac{\mu^{(n)}(x_{\eps})}{n!}b_{\eps}^{n+1}\right|=\left|\frac{\mu^{(n)}(x_{\eps})}{n!}b_{\eps}\right|\frac{1}{\left(b_{\eps}^{-1}\right)^{n}}\le\frac{\bar{C}b_{\eps}}{(b_{\eps}^{-1})^{n}}$,
where $\left|\mu^{(n)}(x_{\eps})\right|\le\int\beta=:\bar{C}$ and
hence $\frac{1}{R}=b$ which is an infinite number. Thereby, in the
particular case when $\frac{1}{R}$ is finite, $f$ is a Colombeau
real analytic function in a neighborhood of $c$. Vice versa, any
Colombeau real analytic function and any ordinary real analytic function
are GRAF.

This characterization also yields the closure of GRAF with respect
to inversion. We first recall that the local inverse function theorem
holds for GSF, see \cite{GK18}. Therefore, if $f\in\gsft(U,\rti)\subseteq\gsf(U,\rti)$
and at the point $x_{0}\in U$ the derivative $f'(x_{0})$ is invertible,
we can find open neighborhoods of $x_{0}\in X\subseteq U$ and of
$y_{0}:=f(x_{0})\in Y$ such that $f|_{X}:X\to Y$ is invertible,
$\left(f|_{X}\right)^{-1}\in\gsf(Y,X)$ and $f'(x)$ is invertible
for all $x\in X$.
\begin{thm}
If $\sigma\le\rho^{*}$ and we use notations and assumptions introduced
above, then $\left(f|_{X}\right)^{-1}\in\gsft(Y,X)$.
\end{thm}

\begin{proof}
For simplicity, set $g:=\left(f|_{X}\right)^{-1}$ and $h(x):=\frac{1}{f'(x)}$
for all $x\in X$, so that $g'(y)=h[g(y)]$ for all $y\in Y$. From
Cor.~\ref{cor:derGRAF} and Thm.~\ref{thm:algProperties}, we know
that $h$ is a GRAF. Therefore, Thm.~\ref{thm:growthDers} yields
$\forall^{0}\eps\,\forall n\in\N:\ \left|h_{\eps}^{(j)}(x_{\eps})\right|\le C_{\eps}\frac{j!}{R_{\eps}^{j}}$
for all $[x_{\eps}]\in B_{s}(x_{0})$ and suitable constants $s$,
$C$, $R\in\rti_{>0}$. For $[y_{\eps}]\in f\left(B_{s}(x_{0})\right)$
(note that this is an open neighborhood of $y_{0}$ because $f$ is
an open map) and these $\eps$, formula (1.15) of \cite[Thm.1.5.3]{KrPa02}
yields $\left|g_{\eps}^{(j)}(y_{\eps})\right|\le j!(-1)^{j-1}{1/2 \choose j}\frac{(2C_{\eps})^{j}}{R_{\eps}^{j-1}}$
for all $j\in\N_{>0}$, and hence $g\in\gsft(U,\rcrho)$ once again
by Thm.~\ref{thm:growthDers}.
\end{proof}
Since $\delta$ is a GRAF, in general the identity principle does
not hold for GRAF. From our point of view this is a feature of GRAF
because it allows to include as GRAF a large class of interesting
generalized functions and hence pave the way to a more general related
Cauchy-Kowalevski theorem. The following theorem clearly shows that
the identity principle does not hold in our framework exactly because
we are in a non-Archimedean setting: every interval is not connected
in the sharp topology because the set of all the infinitesimals is
a clopen set, see e.g.~\cite{GK13b}.
\begin{thm}
\label{thm:identityClopen}Let $U\subseteq\rti$ be an open set and
$f$, $g\in\gsft(U,\rti)$. Then the set
\[
\mathcal{O}:=\text{\emph{int}}\left\{ x\in U\mid f(x)=g(x)\right\} 
\]
is clopen in the sharp topology.
\end{thm}

\begin{proof}
For simplicity, considering $f-g$, without loss of generality we
can assume $g=0$. We only have to show that $\mathcal{O}$ is closed
in $U$. Assume that $c$ is in the closure of $\mathcal{O}$ in $U$,
i.e.
\begin{equation}
c\in U,\ \forall r\in\rti_{>0}\,\exists\bar{c}\in B_{r}(c)\cap\mathcal{O}.\label{eq:closureO}
\end{equation}
We have to prove that $c\in\mathcal{O}$. We first note that for each
$\bar{c}\in\mathcal{O}$, we have $B_{p}(\bar{c})\subseteq\mathcal{O}$
for some $p\in\rti_{>0}$ and hence
\begin{equation}
f(\bar{x})=0\quad\forall\bar{x}\in B_{p}(\bar{c}).\label{eq:ident0}
\end{equation}
Now, fix $n\in\N$ in order to prove that $f^{(n)}(c)=0$. From \eqref{eq:closureO},
for all $r\in\rti_{>0}$ we can find $\bar{c}_{r}\in B_{r}(c)\cap\mathcal{O}$
such that $f^{(n)}(\bar{c}_{r})=0$ from \eqref{eq:ident0}. From
sharp continuity of $f^{(n)}$, we have $f^{(n)}(c)=\lim_{r\to0^{+}}f^{(n)}(\bar{c}_{r})=0$.
Since $f\in\gsft(U,\rti)$ and $c\in U$, we can hence find $\sigma>0$
such that $f(x)=\hypersum{\rho}{\sigma}\frac{f^{(n)}(c)}{n!}\left(x-c\right){}^{n}=0$
for all $x\in B_{\sigma}(c)$, i.e.~$c\in\mathcal{O}$.
\end{proof}
\noindent For example, if $f=\delta$ and $g=0$, the set
\[
\text{int}\left\{ x\in\rti\mid\delta(x)=0\right\} \supseteq\left\{ x\in\rti\mid|x|\gg0\right\} 
\]
is clopen. Thereby, also $\rti\setminus\text{int}\left\{ x\in\rti\mid\delta(x)=0\right\} $
is clopen, and we have
\begin{align*}
\left\{ x\in\rti\mid\delta(x)\ne0\right\}  & \subseteq\rti\setminus\text{int}\left\{ x\in\rti\mid\delta(x)=0\right\} \\
 & \subseteq\left\{ x\in\rti\mid\forall r\in\R_{>0}:\ |x|\le_{\text{s}}r\right\} .
\end{align*}
If we assume that all the derivatives of $f$ are finite and the neighborhoods
of Def.~\ref{def:GRAF} are also finite, then we can repeat the previous
proof considering only standard points $c\in\R$ and radii $r\in\R_{>0}$,
obtaining the following sufficient condition:
\begin{thm}
\noindent \label{thm:identStd}Let $U\subseteq\rti$ be an open set
such that $U\cap\R$ is connected. Let $f$, $g\in\gsft(U,\rti)$
be such that $f|_{V\cap\R}=g|_{V\cap\R}$ for some nonempty subset
$V\subseteq U$ such that $V\cap\R$ is open in the Fermat topology,
i.e.
\[
\forall x\in V\cap\R\,\exists r\in\R_{>0}:\ B_{r}(x)\subseteq V\cap\R.
\]
Finally, assume that all the following quantities are finite:
\begin{enumerate}
\item \label{enu:fg_converge}The neighborhood length $s$ in Def.~\ref{def:GRAF}
is finite for each $c\in U\cap\R$,
\item \label{enu:f_g_finite}$\forall x\in U\,\forall n\in\N:\ f^{(n)}(x)$
and $g^{(n)}(x)$ are finite.
\end{enumerate}
Then $f|_{U\cap\R}=g|_{U\cap\R}$.

\end{thm}

\begin{proof}
The proof proceeds exactly as in Thm.~\ref{thm:identityClopen} but
considering
\[
\mathcal{O}:=\text{int}_{\text{F}}\left\{ x\in U\cap\R\mid f(x)=g(x)\right\} ,
\]
where $\text{int}_{\text{F}}$ is the interior in the Fermat topology
(i.e.~the topology generated by the balls $B_{r}(c)$ for $c\in\rti$
and $r\in\R_{>0}$, see \cite{GK13b}). We have to note that assumption
\ref{enu:f_g_finite} implies that all $f^{(n)}$ are continuous in
this topology (see \cite{GKV}).
\end{proof}
\noindent For example, if $f\in\mathcal{C}^{\omega}(\R)$ is an ordinary
real analytic function and $K$, $h\in\rti$ are finite numbers, the
GRAF $x\in\text{int}(\csp{\rti})\mapsto Kf(hx)\in\rti$, where $\csp{\rti}$
is the set of compactly supported points, satisfies the assumptions
of the last theorem.

\section{Conclusions}

Sometimes, e.g.~in the study of PDE, the class of real analytic functions
is described as a too rigid set of solutions. In spite of their good
properties with respect to algebraic operations, composition, differentiation,
integration, inversion, etc., this rigidity is essentially well represented
by the identity principle that necessarily excludes e.g.~solitons
with compact support or interesting generalized functions. Thanks
to Thm.~\ref{thm:identityClopen}, we can state that this rigidity
is due to the banishing of non-Archimedean numbers from mathematical
analysis. The use of hyperseries allows one to recover all these features
including also interesting non trivial generalized functions and compactly
supported functions. This paves the way for an interesting generalization
of the Cauchy-Kowalevski theorem for GRAF that we intend to develop
in a subsequent work. Its proof can be approached by trying a generalization
of the classical method of majorants, or using the Picard-Lindelöf
theorem for PDE with GSF and then using characterization Thm.~\ref{thm:growthDers}
to show that the GSF solution is actually a GRAF.


\begin{thebibliography}{10}
\bibitem{Ara86}Aragona, J., On existence theorems for the $\bar{\partial}$
operator on generalized differential forms. Proc London Math Soc 53:
474–488, 1986.

\bibitem{Ara96}Aragona, J., Some properties of holomorphic generalized
functions on $\bar{C}$ pseudo-convex domains. Acta Math Hung 70:
167–175, 1996.

\bibitem{BL15}Benci, V., Luperi Baglini, L., A non-archimedean algebra
and the Schwartz impossibility theorem, Monatsh. Math., Vol. 176,
503-520, 2015.

\bibitem{BM19}Berarducci, A., Mantova, V., Transseries as germs of
surreal functions, Transactions of the American Mathematical Society
371, pp. 3549-3592, 2019.

\bibitem{C1}Colombeau, J.F., \emph{New generalized functions and
multiplication of distributions.} North-Holland, Amsterdam, 1984.

\bibitem{CoGa84}Colombeau, J.F., Gale, J.E. Holomorphic generalized
functions. J Math Anal Appl 103: 117–133, 1984.

\bibitem{CoGa88}Colombeau, J.F., Gale, J.E. Analytic continuation
of generalized functions. Acta Math Hung 52: 57–60, 1988.

\bibitem{GaVe11}Garetto, C., Vernaeve, H., Hilbert $\widetilde{\field{C}}$-modules:
structural properties and applications to variational problems. Transactions
of the American Mathematical Society, 363(4). p.2047-2090, 2011.

\bibitem{GK13b}Giordano, P., Kunzinger, M., New topologies on Colombeau
generalized numbers and the Fermat-Reyes theorem, \emph{Journal of
Mathematical Analysis and Applications} 399 (2013) 229–238. http://dx.doi.org/10.1016/j.jmaa.2012.10.005

\bibitem{GK18}Giordano P., Kunzinger M., Inverse Function Theorems
for Generalized Smooth Functions. Chapter in \textquotedbl Generalized
Functions and Fourier Analysis\textquotedbl , Volume 260 of the series
Operator Theory: Advances and Applications pp 95-114, 2017.

\bibitem{GK15}Giordano, P., Kunzinger, M., A convenient notion of
compact sets for generalized functions. Proceedings of the Edinburgh
Mathematical Society, Volume 61, Issue 1, February 2018, pp. 57-92.

\bibitem{GKV}Giordano, P., Kunzinger, M., Vernaeve, H., Strongly
internal sets and generalized smooth functions. Journal of Mathematical
Analysis and Applications, volume 422, issue 1, 2015, pp. 56–71.

\bibitem{TI}Giordano P., Kunzinger M., Vernaeve H., A Grothendieck
topos of generalized functions I: basic theory. See https://www.mat.univie.ac.at/\textasciitilde giordap7/ToposI.pdf

\bibitem{GiLu16}Giordano, P., Luperi Baglini, L., Asymptotic gauges:
Generalization of Colombeau type algebras. Math. Nachr. Volume 289,
Issue 2-3, pages 247–274, 2016.

\bibitem{GKOS}Grosser, M., Kunzinger, M., Oberguggenberger, M., Steinbauer,
R., \emph{Geometric theory of generalized functions}, Kluwer, Dordrecht,
2001.

\bibitem{Han05}Hansheng, Y., Another Proof for the $p$-series Test,
The College Ma thematics Journal, VOL. 36, NO. 3, 2005.

\bibitem{KhKo18}Khalfallah, A., Kosarew, S., Examples of new nonstandard
hulls of topological vector spaces, Proc. Amer. Math. Soc. 146 (2018),
2723-2739.

\bibitem{KhKo18b}Khalfallah, A., Kosarew, S., Bounded polynomials
and holomorphic mappings between convex subrings of $^{*}\mathbb{C}$.
The Journal of Symbolic Logic, 83(1), 372-384, 2018.

\bibitem{Kei63}Keisler, H. J., Limit ultrapowers, Transactions of
the American Mathematical Society 107, pp. 383-408, 1963.

\bibitem{KhSc}Khelif, A., Scarpalezos, D., Zeros of Generalized Holomorphic
Functions, Monatsh. Math. 149, 323–335, 2006. 

\bibitem{Kob96}Koblitz, N., \emph{p-adic Numbers, p-adic Analysis,
and Zeta-Functions}, Graduate Texts in Mathematics (Book 58), Springer;
2nd edition, 1996.

\bibitem{KrPa02}Krantz, S.G., Parks, H.R., \emph{A Primer of Real
Analytic Functions}, Birkhäuser 2002.

\bibitem{MTAG}Mukhammadiev, A., Tiwari, D., Apaaboah, G. et al. Supremum,
infimum and hyperlimits in the non-Archimedean ring of Colombeau generalized
numbers. Monatshefte für Mathematik, Vol. 196, pages 163-190, 2021.

\bibitem{ObPiVa}Oberguggenberger, M., Pilipović, S., Valmorin, V.,
Global Representatives of Colombeau Holomorphic Generalized Functions,
Monatsh. Math. 151, 67–74, 2007.

\bibitem{Ob-Ve}Oberguggenberger, M., Vernaeve, H., Internal sets
and internal functions in Colombeau theory, \emph{J. Math. Anal. Appl.}
341 (2008) 649–659.

\bibitem{PiScVa09}Pilipović, S., Scarpalezos, D., Valmorin, V., Real
analytic generalized functions, Monatsh Math (2009) 156: 85.

\bibitem{Rob73}Robinson, A., Function theory on some nonarchimedean
fields, Amer. Math. Monthly 80 (6) (1973) 87–109; Part II: Papers
in the Foundations of Mathematics.

\bibitem{Sham13}Shamseddine, K., A brief survey of the study of power
series and analyticfunctions on the Levi-Civita fields, Contemporary
Mathematics, Volume596, 2013.

\bibitem{TiGi}Tiwari, D., Giordano, P., Hyperseries in the non-Archimedean
ring of Colombeau generalized numbers. Monatsh. Math. (2021). \url{https://doi.org/10.1007/s00605-021-01647-0}

\bibitem{Ver08}Vernaeve, H., Generalized analytic functions on generalized
domains, arXiv:0811.1521v1, 2008.
\end{thebibliography}
\end{document}